\documentclass[11pt]{amsart}

\usepackage{amsmath,amssymb,amsthm,amsfonts,verbatim}
\usepackage{enumerate}
\usepackage[hidelinks]{hyperref}
\usepackage{graphicx}
\usepackage[utf8]{inputenc}
\usepackage{xcolor}
\usepackage{pinlabel}

\theoremstyle{plain}
\newtheorem{theorem}{Theorem}[section]
\newtheorem{lemma}[theorem]{Lemma}

\newtheorem{corollary}[theorem]{Corollary}
\newtheorem{proposition}[theorem]{Proposition}
\newtheorem{claim}[theorem]{Claim}

\theoremstyle{definition}
\newtheorem{definition}[theorem]{Definition}
\newtheorem{notation}[theorem]{Notation}

\newtheorem{construction}[theorem]{Construction}
\newtheorem{remark}[theorem]{Remark}

\newcommand{\C}{\mathcal{C}}
\newcommand{\RR}{\mathbb{R}}
\newcommand{\ZZ}{\mathbb{Z}}

\newcommand{\QQ}{\mathbb{Q}}
\newcommand{\NN}{\mathbb{N}}

\newcommand{\II}{\mathrm{I\kern -0.11ex I}}

\newcommand{\daggername}{fine curve graph}

\newcommand{\cd}{\mathcal{C}^\dagger}

\DeclareMathOperator{\Homeo}{Homeo}
\DeclareMathOperator{\mcg}{Map}
\DeclareMathOperator{\SL}{SL}
\DeclareMathOperator{\rot}{\rho}
\newcommand{\coloneq}{\mathrel{\mathop:}\mkern-1.2mu=}

\newcounter{notes}%[page]   %Le 2eme argument fait reinitialiser les numeros de notes a chaque page
\title{Rotation sets and actions on curves}

\author[Bowden]{Jonathan Bowden}
\email{jonathan.bowden@ur.de}
\author[Hensel]{Sebastian Hensel}
\email{hensel@math.lmu.de}
\author[Mann]{Kathryn Mann}
\email{k.mann@cornell.edu}
\author[Militon]{Emmanuel Militon}
\email{emmanuel.militon@univ-cotedazur.fr}
\author[Webb]{Richard Webb}
\email{richard.webb@manchester.ac.uk}

\begin{document}

\maketitle

\begin{abstract} 
Building on work of \cite{BHW}, we study the action of the homeomorphism group of a surface $S$ on the {\daggername} $\cd(S)$.  While the definition of $\cd(S)$ parallels the classical curve graph for mapping class groups, we show that the dynamics of the action of $\Homeo(S)$ on $\cd(S)$ is much richer:  homeomorphisms induce parabolic isometries in addition to elliptics and hyperbolics, and all positive reals are realized as asymptotic translation lengths. 

When the surface $S$ is a torus, we relate the dynamics of the action of a homeomorphism on $\cd(S)$ to the dynamics of its action on the torus via the classical theory of {\em rotation sets}.  We characterize homeomorphisms acting hyperbolically, show asymptotic translation length provides a lower bound for the area of the rotation set, and, while no characterisation purely in terms of rotation sets is possible, we give sufficient conditions for elements to be elliptic or parabolic.  
\end{abstract}

\section{Introduction} 

For a closed surface $S$ of genus $g \geq 1$, the classical {\em curve graph} $\C(S)$ has vertex set the isotopy classes of essential simple closed curves on $S$, with edges between pairs of isotopy classes that can be realized disjointly (a slight modification is needed for genus 1 surfaces).   Masur and Minsky \cite{MM1,MM2} showed that the graph $\C(S)$ is Gromov hyperbolic, a result which has become an essential tool in the study of the geometric and algebraic structure of the mapping class group $\mcg(S)$, see for example \cite{BehrstockM,Maher,BKMM,BBF, DGO}.

In \cite{BHW} three of the authors introduced the {\em \daggername } $\cd(S)$ to study the group of all homeomorphisms of $S$. This graph has essential simple closed curves as vertices, so admits a faithful action of $\Homeo(S)$ by isometries.   It is shown in \cite{BHW} that $\cd(S)$ is hyperbolic. This enables  large scale geometric techniques for studying $\Homeo(S)$ via its action on $\cd(S)$, for instance, stable commutator length and fragmentation norm on $\Homeo_0(S)$ are unbounded, answering a question posed by Burago, Ivanov, and Polterovich \cite{BIP}.

In this paper, we show that there is a rich correspondence between the dynamics of the induced action of a homeomorphism on $\cd(S)$ and its dynamics on the surface $S$ itself.  We first address this in a general setting, then specialize to the case of the torus where we study the interactions between the curve graph and the existing dynamical theory of {\em rotation sets} for torus homeomorphisms.   

\subsection{General results} 
Isometries of hyperbolic metric spaces admit a dynamical trichotomy as {\em elliptic}, {\em parabolic} or {\em hyperbolic} according to the asymptotic translation length and diameter of orbits (see Section~\ref{sec:background} for a review of definitions).   In the classical setting of $\C(S)$, it follows from \cite[Proposition~4.6]{MM1} and the Nielsen--Thurston classification  \cite{Thurston} that no mapping classes act parabolically.   By contrast, we show the following.  

\begin{theorem}[Parabolic examples] \label{thm:parabolicsexist}
For any closed orientable surface $S \neq S^2$ there exist isotopically trivial homeomorphisms of $S$ whose action on $\cd(S)$ is parabolic.  
\end{theorem}

The proof of Theorem \ref{thm:parabolicsexist} is via explicit constructions.   
Hyperbolic and elliptic isometries are much easier to build: for the elliptic case, it is easy to define many homeomorphisms which fix a given curve, and many examples of hyperbolics are given in \cite{BHW}.  

Our next result shows that the action of $\Homeo(S)$ on $\cd(S)$ is dynamically very rich; it may be interpreted as giving some justification of the pervasiveness of hyperbolic examples.    

\begin{theorem}[Continuity]  \label{thm:continuity}
Let $S$ have genus $g \geq 1$.  Asymptotic translation length on $\cd(S)$ is a continuous function on $\Homeo(S)$.  
Consequently, 
\begin{enumerate}
\item all nonnegative real numbers can be realized as asymptotic translation lengths in $\Homeo_0(S)$, and 
\item hyperbolicity is an open condition in the $C^0$ topology.
\end{enumerate}
\end{theorem} 

Interestingly, while the set of elliptic mapping classes is not open, one can easily construct examples of open subsets of $\Homeo(S)$ consisting entirely of elliptics, see Construction \ref{const:openelliptics}.  Hence hyperbolicity is not generic in the $C^{0}$ topology.

We highlight another key difference to the classical setting of $\C(S)$ via Theorem~\ref{thm:continuity}. Indeed, Bowditch \cite{Bowditchtight} proved that asymptotic translation lengths on $\C(S)$ belong to $\frac{1}{m}\ZZ$ with $m$ depending only on $S$ i.e. they are uniformly rational. On the other hand  Theorem~\ref{thm:continuity} shows that all nonnegative reals occur as asymptotic translation lengths on $\cd(S)$.

\subsection{Torus homeomorphisms and rotation sets.}

In the case of the torus $T = \RR^2/\ZZ^2$, there is a well-developed theory of the dynamics of isotopically trivial homeomorphisms via their {\em rotation sets}, subsets of $\RR^2$ which, loosely speaking, measure the average displacement of points under iteration.  The influential paper \cite{MZ} of Misiurewicz and Ziemian sparked a general program to relate the dynamics of torus homeomorphisms to the geometric and topological properties of their rotation sets.  
Their work shows that rotation sets are compact and convex, so these are either points, line segments or have nonempty interior.  Each of these do in fact arise:  by considering homeomorphisms of the torus which preserve a foliation by circles and act like a rotation on each of those circles, it is easy to produce examples of homeomorphisms whose rotation set is a singleton or a segment of rational slope containing rational points.	Homeomorphisms whose rotation set is a segment with irrational slope were constructed by Katok (see \cite[Example 1.4]{Handel} and \cite{Kwapisz}); while 
Avila announced a construction of an example whose rotation set is a segment contained in a line of irrational slope and with no rational point.  Le Calvez and Tal \cite{LeCalvezTal} introduced a new orbit forcing theory and proved, among other results, that an irrational slope line rotation set cannot have a rational point in its interior, which verifies a case of the Franks--Misiurewicz conjecture \cite{FranksM}.

We show that the topology of the rotation set classifies homeomorphisms of the torus which act hyperbolically on $\cd(T)$:

\begin{theorem}[Hyperbolic characterisation, special case] \label{thm:hyperbolic-characterisation} 
Let $f \in \Homeo_0(T)$.  The following are equivalent
\begin{enumerate}
\item $f$ acts hyperbolically on $\cd(T)$,
\item $\rot(f)$ has non-empty interior, and
\item there is a finite, $f$-invariant set $P \subset T$ such that the restriction of $f$ to $T - P$ represents a pseudo-Anosov mapping class.
\end{enumerate}  
\end{theorem}

We also show that the area of the rotation set is bounded from below in terms of the square of the asymptotic translation length, see Proposition~\ref{prop:areabound}.

The above theorem highlights a new connection between \textit{geometry} of $\cd(S)$, \textit{dynamics} via the rotation set, and \textit{topology} via the mapping class represented in the complement of an invariant finite set of points. We note that the implication from (2) to (3) above was already known: it is a consequence of the work of Franks \cite{Franks} and Llibre--MacKay \cite{LM}.  The implication from (3) to (2) is more subtle because the rotation vectors of $P$ may be equal. Nonetheless Boyland \cite{Boyland} shows that the Thurston representative on $T-P$ has rotation set with non-empty interior, and so (2) can be deduced from Nielsen fixed point theory and the convexity of the rotation set. In this paper we prove (1) implies (2) and (3) implies (1) using different methods.

Following \cite{Doeff}, one may also define a rotation set for homeomorphisms of $T$ that admit a power isotopic to a (power of a) Dehn twist map.  This set measures the speed of orbits transversally to the twist; see Definition~\ref{def:rot_general}.  
Using this, we extend Theorem \ref{thm:hyperbolic-characterisation} to give an analogous characterisation applicable to all homeomorphisms of $T$.  The precise statement is given in Theorem \ref{thm:hyperbolic-characterisation-general} below.

One might hope that elliptic and parabolic isometries of $\cd(T)$ could similarly be distinguished by their rotation sets.  This is half-true:  we show that one can give sufficient conditions separately for parabolicity and ellipticity, but also show via explicit examples that one cannot hope to give both sufficient and necessary conditions as in Theorem \ref{thm:hyperbolic-characterisation}.   

	\begin{theorem} \label{irrationalimpliesparabolic}
	Let $f\in \Homeo_0(T)$. If $\rho(f)$ is a segment of irrational slope, then $f$ acts on $\cd(T)$ parabolically.   If $\rho(f)$ is a segment of rational slope containing rational points, then $f$ acts elliptically. 
	\end{theorem}
		
Note that there are many elliptic elements whose rotation set is a singleton e.g. any homeomorphism whose support is contained in a disk has null rotation set.  Thus, the condition on elliptics in Theorem \ref{irrationalimpliesparabolic} is not necessary.   
 For parabolics, in Section~\ref{sec:parabolic} we also show

		\begin{proposition}\label{prop:parabolic-with-point-rot2}
          There are homeomorphisms $f$ that act parabolically on $\cd(T)$ with 
           $\rho(f) = \{(0,0)\}$. 
	\end{proposition} 
This implies, in particular, that one cannot hope to distinguish elliptics from parabolics via rotation sets alone.

\subsection{Further questions}
An interesting related question is to study actions on analogous graphs for infinite type surfaces, as a kind of intermediate point between $\cd(S)$ and $\C(S)$.   Work of Bavard \cite{Bavard} and of Horbez--Qing--Rafi \cite{HQR} shows that many of these groups admit dynamically interesting actions on hyperbolic metric spaces.  However, there is no known analog of the Nielsen--Thurston classification, and whether these actions have parabolic elements appears to be open. 

It is natural to try to extend the present work on rotation sets for tori to the case of higher genus surfaces. However, existing analogs of the rotation set on higher genus surfaces  do not lend themselves as easily to such an analysis as we do here.

\subsection{Outline}	
The paper is organised as follows. In Section~\ref{sec:background} we briefly provide the necessary background for our theorems and proofs. In Section~\ref{sec:continuous} we show that the asymptotic translation length of a homeomorphism on $\cd(S)$ is $C^0$-continuous, and therefore $C^r$-continuous for $1\leq r \leq \infty$. In Section~\ref{sec:estimates} we provide useful bounds on distances in $\cd(S)$, which enable us to give statements about rotation sets given actions on curves, or vice versa. These are of independent interest, but also used in our later proofs.  
 In Section~\ref{sec:hyperbolic} we provide the proofs of Theorems~\ref{thm:hyperbolic-characterisation} and \ref{thm:hyperbolic-characterisation-general}. In Section~\ref{sec:parabolic} we provide the proofs of Theorem~\ref{thm:parabolicsexist} and Proposition~\ref{prop:parabolic-with-point-rot2}.

\subsection*{Acknowledgments} Bowden and Hensel greatly acknowledge the support of the Special Priority Programme SPP 2026 Geometry at Infinity funded by the DFG.  Mann was partially supported by NSF CAREER grant DMS 1844516 and a Sloan Fellowship. Militon was supported by the ANR project Gromeov ANR-19-CE40-0007. Webb was supported by an EPSRC Fellowship   EP/N019644/2. 

	\section{Background} \label{sec:background} 
        In this section, we set up notation and recall basic facts from coarse geometry and topological dynamics to be used later in this work.  

		\subsection{Hyperbolic spaces and their isometries}
		Let $X$ be a geodesic metric space. 
		\begin{notation}
			For points $x,y$ in  
			$X$, we
			denote by $[x,y]$ any geodesic between them. We make
			the convention that statements about $[x,y]$ are
			supposed to hold for any choice of geodesic.
		\end{notation}
		We say that $X$ is $\delta$--hyperbolic if geodesic triangles are \emph{$\delta$--slim} i.e.
		if for all points $x,y,z \in Z$ we have
		\[ [x,y] \subset N_\delta([y,z] \cup [z,x]). \]
		Here, $N_\delta$ denotes the closed $\delta$-neighbourhood. We say that $X$ is \textit{Gromov hyperbolic} (or just \textit{hyperbolic})
		if $X$ is $\delta$-hyperbolic for some $\delta\geq 0$.
		For details, background, and basic properties of Gromov hyperbolic spaces, we refer the reader to \cite[Chapter III.H]{BH}.
		It is a straightforward consequence of the definition that the
                Hausdorff distance between any geodesics joining the same two
                points is bounded by $\delta$. 
	
		\bigskip For an isometry $g$ of a Gromov hyperbolic
                space $X$, the \emph{asymptotic
                  translation length} is defined as
		\[ |g|_X \coloneq \lim \limits_{n \to \infty} \tfrac{1}{n}d_X(x, g^n(x)) \]
		It is a standard exercise to see that this limit exists and is independent of $x$. 
		This independence immediately implies
                that the asymptotic translation length is a conjugacy
                invariant of isometries of $X$.
		
		\smallskip We have the following classification of isometries \cite[§ 8]{GromovHG}: 
		\begin{definition}
			Let $g$ be an isometry of a Gromov hyperbolic space. We say that $g$ is
			\begin{description}
				\item[Hyperbolic] if the asymptotic translation length is positive, 
				\item[Parabolic] if the asymptotic translation length is zero  
				but $g$ has no finite diameter orbits, and 
				\item[Elliptic] if $g$ has finite diameter orbits.
			\end{description}
		\end{definition} 
		Note again that these categories are invariant under conjugation by isometries.   The reader may recall that there is an equivalent reformulation of this trichotomy in terms of fixed points on the Gromov boundary of $X$, but we do not require this point of view in the present work.
		
\subsection{Surfaces and Curves}
Throughout this article, \emph{surface} will mean a closed, oriented surface $S$ of genus at 
least $1$, with $T$ being used to denote the torus. 
A \emph{curve} on such a surface will always be required to be an essential closed loop 
(i.e. non-trivial in $\pi_1(S)$), frequently we restrict our attention to simple (i.e. topologically
embedded) curves.  

\smallskip The main geometric tool we use in this article is the following variant of
the classical curve graph, which is the 1-skeleton of the curve complex, see \cite[Section~2.2]{MM1} for a definition.

\begin{definition}  
	For a surface $S$ of genus at least
  two, the {\em \daggername}, denoted by $\cd(S)$, is the graph with
  vertex set equal to the set of essential simple closed curves in
  $S$, and edges given by disjointness. 
  
  In the case of the torus
  $T$, vertices of $\cd(T)$ are again essential simple closed curves, but two
  vertices are now joined by an edge when the corresponding curves are either disjoint or 
  intersect topologically transversely at most once.
\end{definition} 

We endow $\cd(S)$ with a metric in which each edge is isometric to the unit interval, and the distance between two vertices is given by the minimal length of a path between them. The distance between two vertices $\alpha$ and $\beta$ of $\cd(S)$ will be denoted by $d^\dagger(\alpha,\beta)$. 

We need two facts about $\cd(S)$ which are true for any surface $S$ of genus at least $1$
(compare also \cite[Section~5.2]{BHW} for comments on the torus case): 
\begin{theorem}[{\cite[Theorem~3.8]{BHW}}]\label{thm:hyperbolicity}
  The graph $\cd(S)$ is Gromov hyperbolic.
\end{theorem}
\begin{theorem}[{\cite[Lemma~4.2]{BHW}}]\label{thm:TL}
  Suppose that $F\colon S\to S$ is a homeomorphism fixing a finite set $P \subset S$,
  and let $\varphi$ be the mapping class of $S-P$ defined by $F$. Then,
  \[ |F|_{\cd(S)} \geq |\varphi|_{\mathcal{C}(S-P)}. \]

\end{theorem}

By the work of Masur--Minsky (see \cite[Proposition~4.6]{MM1}), if $\varphi$ is pseudo-Anosov, then $|\varphi|_{\mathcal{C}(S-P)}>0$ hence $F$ acts hyperbolically on $\cd(S)$.

\begin{remark}[Smooth versus non-smooth]
In \cite{BHW}, the graph $\cd(S)$ has vertices
corresponding to smooth curves, whereas for the applications here we
need to allow all essential curves as vertices.  As
discussed in \cite[Remark~3.2]{BHW}, these two graphs are quasi-isometric, which
yields Theorem~\ref{thm:hyperbolicity}.  The argument
given in \cite{BHW} for Theorem~\ref{thm:TL} is stated in the smooth context but applies equally well in the $C^0$ setting, giving the statement above.
\end{remark} 

\smallskip In addition, we use the following easy lemma:
\begin{lemma}\label{lem:intersection-bound}
  Suppose that $S$ is any surface of genus $g \geq 1$, and
  suppose that $\alpha$ and $\beta$ are two curves on $S$ intersecting in
  a finite number of points. Then
  \[ d^\dagger(\alpha, \beta) \leq 2\#(\alpha \cap \beta) + 2 \]
\end{lemma}
This lemma can be deduced easily from Lemma~3.4 of \cite{BHW} and
the corresponding estimate for usual curve graphs.  We give an alternative proof using 
 \emph{curve surgery}, as a warm-up for later arguments which will use similar tools.    
 
\begin{proof} 
Suppose that $\alpha$ and $\beta$ are two curves which intersect  
in a finite number of points. Let $a\subset \alpha$ be an embedded  
subarc such that $a \cap \beta$ is equal to the endpoints of $a$. Denote by
$b', b'' \subset \beta$ the two connected components of $\beta \setminus a$.  Then
$a \cup b'$ and $a \cup b''$ are simple closed curves, at least one of
which is essential.  Call this essential curve $\beta'$.  

If $a$ approaches $\beta$ from the same side at both endpoints, then $\beta'$ is homotopic to, and disjoint from,
a curve $\beta''$ which is disjoint from $\beta$.
If instead $a$ approaches $\beta$ from opposite sides at both 
endpoints, one can instead find such a curve $\beta''$ that intersects $\beta$ in a single point.  

Thus, this curve $\beta''$ obtained from $\beta$ by surgery along $a$ 	
has distance at most $2$ from $\beta$ in $\cd(S)$.
By construciton, $\beta''$ can be taken to intersect $\alpha$ in strictly
fewer points than $\beta$, which shows Lemma~\ref{lem:intersection-bound}
by induction. 
\end{proof} 

\begin{remark} The bound given in Lemma \ref{lem:intersection-bound} is far from optimal.  In fact, one could also bound the distance $d^\dagger$ by the logarithm of the intersection number (as in the case of usual curve graphs, see \cite{Hempel}). Namely, by choosing the right subarc $a$, one can ensure that both choices for $\beta''$ are essential, so we can halve the number of intersections in each step.
\end{remark}

\subsection{Rotation sets for torus homeomorphisms} \label{sec:rot}
Here and in what follows, we view the torus with its standard Euclidean structure $T = \RR^2/\ZZ^2$.  

\begin{definition} 
Let $\tilde{f} \in \Homeo_0(\RR^2)$ be a lift of an isotopically trivial homeomorphism of $T$. The {\em rotation set} $\rot(\tilde{f}) \subset \RR^2$ is the set of vectors 
\[ 
\rot(\tilde{f})\coloneq \left\{ v \in \RR^2 : \exists x_i \in \RR^2, n_i \to \infty \ \mathrm{s.t.} \ \tfrac{(\tilde{f}^{n_i}(x_i) - x_i)}{n_i} \to v \right\}.
\]
\end{definition}

We recall some basic properties of $\rho$.  The following are easy consequences of the definition:
\begin{enumerate} 
\item For $p \in \ZZ^2$ and $n \in \ZZ$, we have 
\[ \rot(\tilde{f}^n + p) =  n\rot(\tilde{f}) + p \] 
\item For any lift $\tilde{g}$ of an isotopically trivial torus homeomorphism, we have 
\[ \rot(\tilde{g} \tilde{f} \tilde{g}^{-1}) = \rot(\tilde{f}), \] and 
\item For any matrix $A \in SL_2(\mathbb{Z})$, we have
\[ \rot(A \tilde{f} A^{-1}) = A(\rot(\tilde{f})). \]
\end{enumerate}

Another important and much less trivial property is the result of Misiurewicz and Ziemian \cite{MZ} that $\rho(\tilde{f})$ is compact, convex, and equal to the convex hull of the {\em pointwise rotation set} 
\[ 
\rho_p(\tilde{f}):= \left\{ v \in \RR^2 : \exists x \in \RR^2, n_i \to \infty \ \mathrm{s.t.} \ \tfrac{(\tilde{f}^{n_i}(x) - x)}{n_i} \to v \right\}.
\]

Thus, $\rot(\tilde{f})$ is either a point, a closed interval, or has non-empty interior.  These, and many other properties of interest (for example, the property of containing a point with rational coordinates) are invariant under integer translations and therefore independent of the lift of $f$ chosen.  When considering such questions, we will often abuse notation and simply write $\rot(f)$, rather than $\rho(\tilde{f})$, thinking of $\rot(f)$ as a set well defined up to translation by $\ZZ^2$.

As a byproduct of the proof of the convexity of the rotation set, Misiurewicz and Ziemian obtained the following result that we will need. Fix a fundamental domain $D \subset \RR^2$. For any integer $n>0$, let
$$\frac{1}{n}\tilde{f}^{n}(D)=\left\{ \frac{\tilde{f}^{n}(x)}{n} : x \in D \right\}.$$
 
 \begin{lemma}[Misiurewicz-Ziemian \cite{MZ}] \label{lem:MZ}
 The sequence of compact subsets $(\frac{1}{n}\tilde{f}^{n}(D))$ converges to $\rho(\tilde{f})$ in the Hausdorff topology.
 \end{lemma}

As mentioned in the introduction, there is a rich and well-developed theory relating the dynamics of torus homeomorphisms to the geometric and topological properties of their rotation sets.  A general introduction can be found in the original work of Misiurewicz and Ziemian \cite{MZ}, and a more detailed description of recent developments in the survey \cite{Beguin} (in French).   We will need the following two important results.  

\begin{theorem}[Franks \cite{Franks}] If $(a/q, b/q)$ is a point with rational coordinates in the interior of $\rot(\tilde{f})$, then $f$ has a periodic point of period $q$.
\end{theorem}   

\begin{theorem}[Llibre--MacKay, \cite{LM}] \label{thm:LM}
Suppose $f \in \Homeo_0(T)$ has rotation set with non-empty interior.  Then there exists a finite $f$-invariant set $P$ such that the restriction of $f$ to $T - P$ is a pseudo-Anosov mapping class.  
\end{theorem}

The existence of the finite set $P$ comes from Franks's theorem quoted above. Llibre--Mackay's insight was to show that the orbits of any three periodic points with non-collinear rational rotation vectors gives the desired set $P$ with the pseudo-Anosov property.

\subsection*{Dehn twists} 
As discussed in \cite{Doeff}, homeomorphisms of $T$ isotopic to a Dehn twist map also have a rotation set relative to the direction of the twist.  We explain this now.  
\begin{definition}
Let $\alpha$ be a curve of $T$. We say a homeomorphism $f$ of $T$ is a \textit{Dehn twist map around} $\alpha$ when the following hold
\begin{enumerate}
\item the support of $f$ is contained in an embedded annulus $A\subset T$ which contains $\alpha$, and
\item there exists a chart $\varphi\colon A \rightarrow \mathbb{R} / \mathbb{Z}\times [0,1]$ which sends $\alpha$ to $\mathbb{R} / \mathbb{Z} \times \left\{ 1/2 \right\}$ and $r \in \mathbb{Z} \setminus \left\{ 0 \right\}$ such that $\varphi f_{|A} \varphi^{-1}= \tau^r$, where
$$\begin{array}{rrcl}
\tau\colon & \mathbb{R} /\mathbb{Z} \times [0,1] & \rightarrow & \mathbb{R} /\mathbb{Z} \times [0,1] \\
 & (x,y) & \mapsto & (x+y,y)
 \end{array} .$$ 
\end{enumerate}
\end{definition}
Fix a simple (oriented) essential loop $\alpha$ and let $\check{T}$ denote the associated cyclic cover of $T$. We fix an identification of $\check{T}$ with $\mathbb{R} / \mathbb{Z} \times \mathbb{R}$, with deck group equal to the integral translations in the second coordinate, oriented so that a lift of $\alpha$ is oriented in the positive direction.   Let $f \in \Homeo(T)$ be isotopic to a Dehn twist map around $\alpha$, and let 
 $\check{f}\colon \check{T} \rightarrow \check{T}$ be a lift of $f$.    Let $p_2 \colon \check{T}= \mathbb{R}/\mathbb{Z} \times \mathbb{R} \rightarrow \mathbb{R}$ be the projection to the second co-ordinate, and 
$$\begin{array}{rrcl}
t \colon & \check{T} & \rightarrow & \check{T} \\

 & (x,y) & \mapsto & (x,y+1)
 \end{array}.$$
\begin{definition} \label{def:rot_general} \cite{Doeff}
The rotation set $\rho_\alpha(\check{f})$ of the lift $\check{f}$ is the subset of $\mathbb{R}$ consisting of accumulation points of 
$$ \left\{ \frac{p_{2}(\check{f}^{n}(\check{x}))-p_{2}(\check{x})}{n}: n \geq 1 \ \text{and} \ \check{x} \in \check{T} \right\}.$$
\end{definition}

This set is a segment of $\mathbb{R}$. Similarly to the case of homeomorphisms isotopic to the identity, it follows from the definition that $\rho_\alpha(f) = \rho_{g\alpha}(gfg^{-1})$ for any homeomorphism $g$ of $T$, and 
 for any integers $p$ and $q$, we have
$$ \rho(t^p \check{f}^q)=q\rho(\check{f})+p.$$ 
Thus, two lifts of $f$ to $\check{T}$ have rotation sets which differ by an integral translation of $\mathbb{R}$, and so we use $\rho_\alpha(f)$ (and occasionally $\rho(f)$ when $\alpha$ is understood), rather than $\rho_\alpha(\check{f})$ to mean the rotation set of any fixed lift when we wish to speak of properties independent of choice of lift.  
		
%----------		
\section{Asymptotic Translation Length} \label{sec:continuous} 
In this section we prove continuity of asymptotic translation length.  As a direct consequence, one may conclude that hyperbolicity for the action on $\cd(S)$ is an open condition on $\Homeo(S)$.  Before embarking on the proof, we take a brief detour to discuss the contrasting results that there are hyperbolic elements arbitrarily close to the identity, as well as open sets of elliptic elements. 

Here, and for the remainder of the paper, $|\cdot|$ always denotes
asymptotic translation length on the curve graph $\cd(S)$.

\begin{lemma}\label{lem:pa-root}
	For any closed surface $S$ of genus at least $1$, 
	any $C^0$ neighbourhood of the identity contains homeomorphisms acting hyperbolically on $\cd(S)$.
\end{lemma}
\begin{proof} First assume that $S$ has genus at least $2$.  Recall that a \textit{filling} loop $\gamma$ means that every loop homotopic to $\gamma$ intersects every essential simple closed curve on $S$. 
	Let $p \in S$ be a point, and let $\gamma\colon [0,1]\to S$ be a smooth, filling loop based at $p$, with transverse self-intersections. 
	
	Next, choose times $0=t_0<\ldots<t_n=1$ so that $\gamma\vert_{[t_i,t_{i+1}]}$ has length at most $\epsilon$.
	By possibly adding more times before self-intersection points, we may assume that $\gamma$ can be covered
	with disks $D_1, \ldots, D_{2k}$ so that
	\begin{enumerate}
		\item Each $D_i$ has diameter $<\epsilon$,
		\item $D_i \cap D_j = \emptyset$ if $i-j$ is even, and
		\item $\gamma\vert_{[t_i,t_{i+1}]} \subset D_i$ for all $i$.
	\end{enumerate}
	We let $G_0$ (respectively, $G_1$) be the endpoint of the isotopy supported in the union of all $D_i$ for $i$ even (respectively, odd) and in each such $D_i$ slides $\gamma(t_i)$ to $\gamma(t_{i+1})$ along $\gamma\vert_{[t_i,t_{i+1}]}$. By choosing $\epsilon$ small enough,
	\[ F = G_1\circ G_0 \]
	is arbitrarily close to the identity. We claim that $F$ acts on $\cd(S)$ hyperbolically. To this end, it suffices
	to show that some power of $F$ does. Observe that $F^{k}$ is a homeomorphism of $S$ fixing $p$, and (by composing the isotopies defining $G_0, G_1$), there is an isotopy 
	from $F^{k}$ to $\mathrm{id}$ whose trace of the point $p$ is exactly the loop $\gamma$. This implies that the isotopy class $[F^{k}]$
	lies in the kernel of the forgetful map
	\[ 1 \to \pi_1(S,p) \to \mathrm{Mcg}(S-p) \to \mathrm{Mcg}(S) \to 1 \]
	and corresponds exactly to the loop $\gamma\in \pi_1(S,p)$ (compare \cite[Section~4.2]{Primer} for this Birman exact sequence).
	Since $\gamma$ is filling, Kra's theorem \cite{Kra} implies that the mapping class $[F^{k}] \in \mathrm{Mcg}(S-p)$
	is pseudo-Anosov. Theorem~\ref{thm:TL} then implies that $F^{k}$ (hence $F$) acts hyperbolically on $\cd(S)$. 

	If $S = T$ is the torus, we need to choose $\gamma$ to be filling in $T\setminus\{q\}$ for some point $q$ (as the Birman exact sequence requires at least one puncture in the case of the torus), and argue analogously.
\end{proof}
\begin{remark}
	We emphasise that the construction in Lemma~\ref{lem:pa-root} is flexible -- the choice of the 
	filling loop $\gamma$ is arbitrary, and we can construct a root of the corresponding point-pushing pseudo-Anosov.
\end{remark}

The following gives a general construction of open sets of elliptic elements which can be taken arbitrarily close to (though not containing) the identity.  

\begin{construction}[Open sets of elliptics] \label{const:openelliptics}
Let $A \subset S$ be an embedded essential, closed annulus, and $f: S \to S$ a homeomorphism isotopic to the identity and sending $A$ into the interior of $A$.  Let $\alpha$ denote a boundary curve of $A$.   There exists a neighborhood of $f$ in $\Homeo_0(S)$ consisting of homeomorphisms which send $A$ into the interior of $A$.  Any such homeomorphism $g$ will satisfy that $g^N(\alpha) \cap \alpha = \emptyset$, giving a bounded orbit. 
\end{construction}

%------------------------

\subsection{Proof of Theorem \ref{thm:continuity}}
We now prove the following result.  
	\begin{theorem} \label{thm:continuity-TL}
		Let $f_m \to f$ in the $C^0$-topology on $\Homeo(S)$. Then \[|f_m| \to |f|.\] In particular, the set of
		homeomorphisms acting hyperbolically on $\cd(S)$ is an open set in the $C^0$-topology.
	\end{theorem}
This is the main technical content of Theorem \ref{thm:continuity} stated in the introduction, indeed we have: 

\begin{proof}[Proof of Theorem \ref{thm:continuity} given Theorem \ref{thm:continuity-TL}]
Openness of hyperbolicity is immediate from the continuity of asymptotic translation length.  
What remains to show is that all positive real values are realized.  In \cite{BHW}, it is shown that taking powers of elements which act as pseudo-Anosov homeomorphisms relative to a fixed finite set produces elements with arbitrarily large translation length.  Since $\Homeo_0(S)$ is connected, this implies all positive real values are attained.  
\end{proof} 

	The proof of Theorem \ref{thm:continuity-TL} will occupy the rest of the section. 
        We begin with the following elementary lemma.
        \begin{lemma}\label{lem:distance-continuity}
          Suppose that $g_m$ is a sequence of homeomorphisms converging to a homeomorphism
          $g$ in the $C^0$-topology, and let $\alpha$ be an essential simple closed curve. Then, for all
          $m$ large enough, we have
          \[ d^\dagger(g_m\alpha, g\alpha) \leq 2. \]
        \end{lemma}
        \begin{proof}
          Let $C$ be a small collar neighbourhood of $g\alpha$, so that the boundary
          of $C$ consists of two simple closed curves disjoint from $g\alpha$. Then, for
          all $m$ large enough, we have
          \[ g_m\alpha \subset \mathrm{int}(C), \]
          which implies the claim.
        \end{proof}
        
        We now pick a sequence $f_m$ with limit $f$ as in Theorem~\ref{thm:continuity-TL}.
	\begin{lemma} \label{lem:transl_length}
          Given any $n\in \mathbb{N}$ and essential simple closed curve
          $\alpha$, for all sufficiently large $m$ we have 
          \[ n|f_m|=|f_m^n| \leq d^\dagger(\alpha, f^n \alpha) +2.\]
	\end{lemma}
	\begin{proof}
          The first equality is immediate from the definition of asymptotic translation length.
          For the inequality, simply apply Lemma~\ref{lem:distance-continuity}
          taking $g_m = f_m^n$ and $g=f^n$. 	\end{proof}
	As a consequence of this lemma we have 
	\[ \limsup_{m\rightarrow \infty} |f_m| \leq |f| \]
	In particular, if $|f|=0$, we have $\lim_{m\rightarrow\infty} |f_m| = 0 = |f|$ and we are done. 
        Hence, we may from now on assume that $|f| > 0$ and we aim to prove
	that $\liminf_{m\rightarrow\infty} |f_m| \geq |f|$.
	
	Let $\delta$ be a hyperbolicity constant for $\cd(S)$, so that 
        any geodesic triangle in $\cd(S)$ is $\delta$--slim, and any geodesic
        quadrilateral is $2\delta$--slim.

	The first step will be to choose a convenient curve $\alpha$ to serve as a basepoint.
	\begin{lemma}\label{lem:tunnel} 
		There is $\alpha\in\cd(S)$ and $N\in\mathbb{N}$ such that for all $n\geq N$ we have 
		\[ d^\dagger( \alpha, [f^{-n} \alpha, f^n \alpha] ) \leq 2\delta, \]
		and furthermore, for sufficiently large $m$,
		\[ d^\dagger( \alpha, [f_m^{-N} \alpha, f_m^N \alpha] ) \leq 2\delta. \]
	\end{lemma}
	\begin{proof} 
          Start with any $\beta\in\cd(S)$. Then the sequence $(f^n \beta)_n$ is a
          $C$-quasi-geodesic for some $C>0$.   By  \cite[Theorem III.H.1.7.]{BH} (quasi-geodesics fellow travel geodesics) there exists a constant $B = B(\delta, C)$ such that any geodesic segment $[f^i \beta, f^j\beta]$ lies within Hausdorff distance $B$ of 
          $( f^k \beta
          )_{ i \leq k \leq j}$.

          Take any $N \in \NN$ satisfying 
          \[N > \max \left\{ \frac{3B+3\delta}{|f|}, \frac{2\delta+B+4}{|f|} \right\}  \] and choose
          $\alpha$ to be a closest-point projection of $\beta$ to
          $[f^{-N}\beta,f^N\beta]$. We have that $d^\dagger(\beta,\alpha)\leq
          B$. Fix $n > N$ and consider a geodesic segment $[f^{-n}\alpha,f^n\alpha]$ with $n\geq
          N$. Now we study the geodesic quadrilateral given by the
          aforementioned two geodesics together with $[f^{-N}\beta,
          f^{-n}\alpha]$ and $[f^N\beta, f^n\alpha]$, as indicated in the left
          side of Figure~\ref{fig:33-34figures}.

          \begin{figure}
            \centering
            \includegraphics[width=\textwidth]{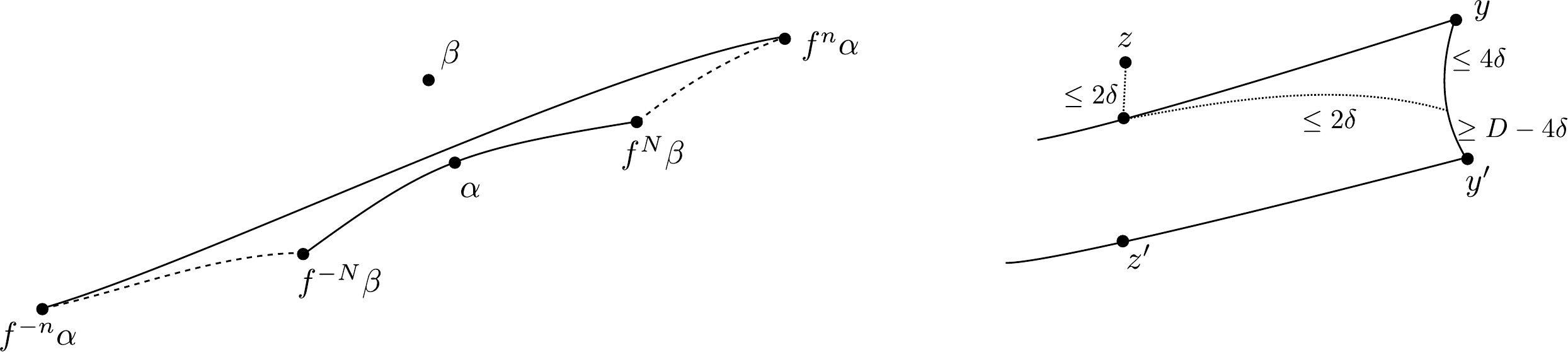}
            \caption{The situation in the proof of Lemma~\ref{lem:tunnel} (left) and~\ref{lem:4delta} (right)}
            \label{fig:33-34figures}
          \end{figure}

Since
          quadrilaterals are $2\delta$--slim, the point $\alpha \in
          [f^{-N}\beta,f^N\beta]$ has distance at most $2\delta$ to
          one of the other three sides.  To prove the lemma we want to
          show that $\alpha$ is not in the $2\delta$-neighborhood of
          $[f^{-N}\beta, f^{-n}\alpha]$ or of $[f^N\beta, f^n\alpha]$.

          So, suppose for contractiction that $\alpha$ is within
          $2\delta$ of $[f^N\beta, f^n\alpha]$ (the other case is
           analogous). Since $d^\dagger(\alpha, \beta) \leq B$, we have
          \[ [f^N\beta, f^n\alpha] \subset N_{B+\delta}([f^N\beta,
          f^n\beta]) \] Therefore, $[f^N\beta, f^n\alpha]$ is in a
          $(2B+\delta)$-neighborhood of $(f^i\beta)_{N\leq i \leq
            n}$. Since $\alpha$ is within $2\delta$ of $[f^N\beta,f^n\alpha]$,
	  this shows that $\beta$ has distance at most
          $3B+3\delta$ to a point in $(f^i\beta)_{N\leq i \leq n}$,
          This contradicts the first lower bound in our choice of $N$, since $ |i|\cdot|f| \leq d^\dagger(f^i\beta,\beta) \leq 3B+3\delta$, but on the other hand we have $N\leq i $ by definition of $i$.
	
          \medskip Now we tackle the last claim of the lemma, using the fact that \[N > \frac{2\delta+B+4}{|f|}.\] Pick
          $\alpha$ as before. Then for sufficiently large $m$ for
          $i\in \{-N,N\}$ we have that $d^\dagger(f_m^i\alpha,f^i\alpha)\leq 2$. We
          consider the geodesic quadrilateral with vertices
          $f_m^{-N}\alpha$, $f^{-N}\beta$, $f^N\beta$, and
          $f_m^N\alpha$. Recall that $\alpha$ lies on the geodesic
          $[f^{-N}\beta,f^N\beta]$. As before, if
          $\alpha$ is within $2\delta$ of
          $[f^N\beta,f_m^N\alpha]$, then $\alpha$ is within $2\delta +
          B+4$ of $f^N\alpha$, which is a contradiction.
	\end{proof}

	\begin{lemma}\label{lem:4delta}
          If $N$ is chosen large enough, and $m$ is sufficiently large (depending on $N$), then for all $n \in \NN$ and
          $i$ with $-n\leq i \leq n$, we have that $(f_m^N)^i\alpha$ is
          within $4\delta$ of $[(f_m^N)^{-n}\alpha,(f_m^N)^n\alpha]$.
	\end{lemma}
	
	\begin{proof} 
          Here, we require in addition to the constraints of the
          previous proof that $N>\frac{8\delta+2}{|f|}.$

          Given $n \in\NN$, fix a geodesic segment $[(f_m^N)^{-n}\alpha,(f_m^N)^n\alpha]$ and let $i$ be an index with $-n\leq i \leq n$ such that 
          that $z \coloneq (f_m^N)^i\alpha$ maximises distance to
          $[(f_m^N)^{-n}\alpha,(f_m^N)^n\alpha]$.  Let $D$ denote the distance from $z$ to $[(f_m^N)^{-n}\alpha,(f_m^N)^n\alpha]$.

          Let $x\coloneq (f_m^N)^{i-1}\alpha$ and let $y \coloneq
          (f_m^N)^{i+1}\alpha$. Let $z',x'$ and $y'$ be choices of 
          closest-point projections of $z,x$ and $y$, respectively, to
          $[(f_m^N)^{-n}\alpha,(f_m^N)^n\alpha]$.  
	  Finally, let $\overline{z}$ be a closest point projection of $z$ to $[x,y]$, and recall that 
          $z$ is within $2\delta$ of $[x,y]$ by
          Lemma~\ref{lem:tunnel}. 
          
          Now consider the geodesic quadrilateral $[x,y]$, $[y,y']$,
          $[y',x']$, $[x',x]$.  By $2\delta$--slimness,
          $\overline{z}\in[x,y]$ is $2\delta$--close to one of the
          other three sides. To finish the proof, what we need to show is that the distance from $\overline{z}$  to $[x', y']$ is at most
          $2\delta$.  Similarly to the previous lemma, we proceed by using a proof by contradiction to show that $\overline{z}$  cannot be close to one of the other two sides.  

          So, for contradiction, assume that $\overline{z}$ is $2\delta$--close to
          a point $\overline{y}\in[y,y']$ (the case of $[x,x']$ is analogous).  The set-up is illustrated in the right side of Figure~\ref{fig:33-34figures}.

          We then have
          \[ D = d(z, [(f_m^N)^{-n}\alpha,(f_m^N)^n\alpha]) \leq 4\delta + d(\overline{y},  [(f_m^N)^{-n}\alpha,(f_m^N)^n\alpha])), \]
          and therefore
          \[ d^\dagger(\overline{y},  [(f_m^N)^{-n}\alpha,(f_m^N)^n\alpha])) \geq D - 4\delta. \]
          On the other hand, because our choice of $z$ was a point that maximised distance to $[(f_m^N)^{-n}\alpha,(f_m^N)^n\alpha]$, we have that
          \[ d^\dagger(y, [(f_m^N)^{-n}\alpha,(f_m^N)^n\alpha])) \leq D, \]
          hence
          \[ d^\dagger(\overline{y}, y) \leq 4\delta. \]
          Thus we obtain in total that
          \[ d^\dagger((f^N_m)^i\alpha, (f^N_m)^{i+1}\alpha) = d(z, y) \leq 8\delta, \]
          hence
          \[ d^\dagger(\alpha, f^N\alpha) \leq 8\delta + 2. \]
          This contradicts $N>\frac{8\delta+2}{|f|}$, completing the proof.          
                  \end{proof}
	
	\begin{remark} \label{rem:4delta}
        Observe that the proof of Lemma \ref{lem:4delta} also shows that
        for any large enough $N$, the closest-point projections of the
        $(f_m^N)^i\alpha$ for $-n\leq i \leq n$ on the geodesic
        $[(f_m^N)^{-n}\alpha,(f_m^N)^n\alpha]$ are monotonic. 
        \end{remark} 
	
	\begin{lemma}  For all sufficiently large $N$, and sufficiently large $m$ (depending on $N$),  for all $n\in\mathbb{N}$ we have
		\[2n(|f^N|-2-8\delta) \leq d^\dagger((f_m^N)^{-n}\alpha, (f_m^N)^n\alpha).\] 
		Therefore, $|f_m| \geq |f|-\frac{2+8\delta}{N}$.
	\end{lemma}
	
	\begin{proof} 
          Divide the segments of a geodesic
          $[(f_m^N)^{-n}\alpha,(f_m^N)^n\alpha]$ into $2n$ disjoint pieces
          by taking  closest-point projections of the points $(f_m^N)^i\alpha$
          to the geodesic (by the monotonicity of projections guaranteed by Remark~\ref{rem:4delta}).  Each
          piece has length at least $|f^N|-2-8\delta$ by
          Lemma~\ref{lem:4delta} and Lemma~\ref{lem:transl_length}
          so we obtain the first
          inequality. The second inequality follows by rearranging the
          first inequality and taking the limit of this quantity over $n$ as $n$ tends to $\infty$.
        \end{proof}
	
        As a consequence of the above lemma, we obtain
        \[ \liminf_{m\rightarrow\infty} |f_m| \geq |f|-\frac{2+8\delta}{N} \]
        for all sufficiently large $N$. This shows that $\liminf_{m\rightarrow\infty} |f_m| \geq |f|$, and
        Theorem~\ref{thm:continuity-TL} follows. \qed

	\section{Distance Estimates} \label{sec:estimates}
	In this section, we provide key results which allow us to connect the geometry of $\cd(S)$ to the topology of curves on $S$ and their lifts to specific covers.  
			
         Our first criterion, Lemma \ref{lem:cover-distance-2} below, works for all surfaces $S$ of genus $g
         \geq 2$, and uses covering spaces. One can think of this as a
         $\cd(S)$--version of the criterion introduced by Hempel in
         \cite[Section 2]{Hempel}.  
 To state it, we need to introduce some vocabulary and basic observations.   If $\pi\colon S' \to S$ is a (possibly branched)
         cover, and $\alpha \subset S$ is an essential simple closed curve (disjoint from the branch points), then we say that 
         an \emph{elevation} of $\alpha$ is a connected component of
         $\pi^{-1}(\alpha)$. 
         The following two properties are obvious, yet important:
        \begin{enumerate}
        \item Elevations of simple closed curves are simple. If the
          cover has finite degree, an elevation of such a curve is also closed.  
        \item Any two elevations of disjoint simple closed curves are disjoint.
        \end{enumerate}
        
	\begin{lemma}\label{lem:cover-distance-2}
          Let $S$ be a surface of genus $g \geq 2$. If $\alpha$ and $\beta$
          are two curves with $d^\dagger(\alpha, \beta) = 2$, then
          there is a degree $2$ cover $X \to S$ such that $\alpha$ and $\beta$ admit disjoint elevations in $X$.
	\end{lemma}
	\begin{proof}
          Since $d^\dagger(\alpha, \beta) = 2$, there is a curve
          $\gamma$ which is disjoint from both $\alpha$ and $\beta$.  
          The following argument shows that we 
          may take such a curve to be non-separating.  Since $d^\dagger(\alpha, \beta) \neq 1$, the curves intersect, and so if 
          $\gamma$ were separating, then
          $\alpha, \beta$ would be contained in the same complementary
          component of $\gamma$.  Thus, we could replace $\gamma$ with any non-separating curve in the other
          complementary component.
		
          Taking the mod-2 intersection number with $\gamma$ gives a homomorphism $\pi_1(S) \to \ZZ/2\ZZ$; 
          let $X$ be the associated degree $2$ cover. We claim that $X$ has
          the desired property, and in fact $\alpha, \beta$ even admit 
          disjoint lifts to $X$.  Let $\gamma_1$ and $\gamma_2$ denote the two elevations of $\gamma$ to $X$.  Then 
          \[ X - (\gamma_1\cup\gamma_2) = X_1 \cup X_2 \] 
          where 
          $X_1$ and $X_2$ are the two preimages of $S-\gamma$.  Now, the
          desired elevations can be obtained by lifting of $\alpha$
          into $X_1$, and $\beta$ into $X_2$.
	\end{proof}
        Iterating Lemma~\ref{lem:cover-distance-2} immediately yields the following.
	\begin{corollary}\label{cor:cover-distance-bound}
          If $S$ is a surface of genus $g \geq 2$ and $\alpha, \beta \in \cd(S)$ with $d^\dagger(\alpha, \beta)\leq n$, then there is a cover of $S$
          of degree at most $2^n$ to which $\alpha, \beta$ admit
          disjoint elevations.
	\end{corollary}
	As a consequence, we have the following criterion that we will frequently apply later.
	\begin{lemma}\label{lem:small-diameter-one-cover}
		Let $S$ be a surface of genus $g \geq 2$ and let $K\geq 0$. There
		is a finite-sheeted cover $X \to S$ (depending only on $K$) such that 
		any two curves $\alpha$ and $\beta$ on $S$ satisfying $d^\dagger(\alpha,\beta)\leq K$ admit disjoint elevations to $X$.
	\end{lemma}
	\begin{proof}
                
		Let $\Gamma < \pi_1(S)$ be the intersection
		of the finite index subgroups of degree at most $2^K$ in $\pi_1(S)$.  
		It is a well-known fact that a finitely generated group has
		only finitely many subgroups a given index (e.g. since an index $n$ subgroup $H$ of $G$ is determined by the
		action of $G$ on the cosets of $H$, and there are only finitely many homomorphisms of a finitely generated group to the symmetric group of $n$ elements).  Thus, $\Gamma$ also has finite index in $\pi_1(S)$.    
		By covering space theory, the subgroup $\Gamma$ determines a finite cover $X \to S$ of degree equal to the index of $\Gamma$.  
		We claim this cover has the desired property.  Indeed, by Corollary~\ref{cor:cover-distance-bound}, any pair of curves $\alpha$ and $\beta$ satisfying $d^\dagger(\alpha,\beta)\leq K$ will admit disjoint elevations to some cover $X'$, where $X' \to S$ has degree at most $2^K$. By definition of $\Gamma$, we have $\Gamma < \pi_1(X') < \pi_1(S)$, and hence $X\to S$ factors through $X \to X' \to S$. Disjoint elevations stay disjoint in further covers, hence $\alpha$ and $\beta$ have disjoint elevations to $X$ as required.  
	\end{proof}

	\subsection{A crossing number estimate}
	We also develop an upper bound on distance in $\cd(T)$ (Lemma~\ref{lem:crossing-estimate}, below) 
	specifically for the torus.  This will be used in our work on rotation sets in the next section.  
   
	\begin{definition}
          Let $\alpha$ and $\beta$ be essential simple closed curves on the torus
          $T=S^1\times S^1$. Interpret $\beta$ as a map $\beta \colon [0,1] \to T$
          and denote by $\widetilde{\beta}\colon [0,1]\to\tilde{T}=\RR^2$ a lift.

          We define the {\em $\alpha$-crossing number}
          $C_\alpha(\beta)$ as the number of distinct elevations of $\alpha$
          which $\widetilde{\beta}$ intersects.
	\end{definition}
         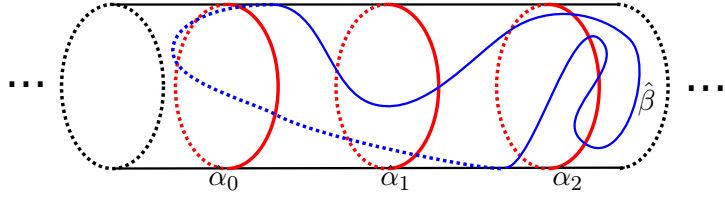
\begin{figure}
           \labellist 
  \hair 2pt
\pinlabel $\hat{\alpha}_0$ at 150 -5
\pinlabel $\hat{\alpha}_1$ at 270 -5
\pinlabel $\hat{\alpha}_2$ at 390 -5
\pinlabel $\hat{\beta}$ at 445 50
   \endlabellist
          \centering
          \includegraphics[width=0.75\textwidth]{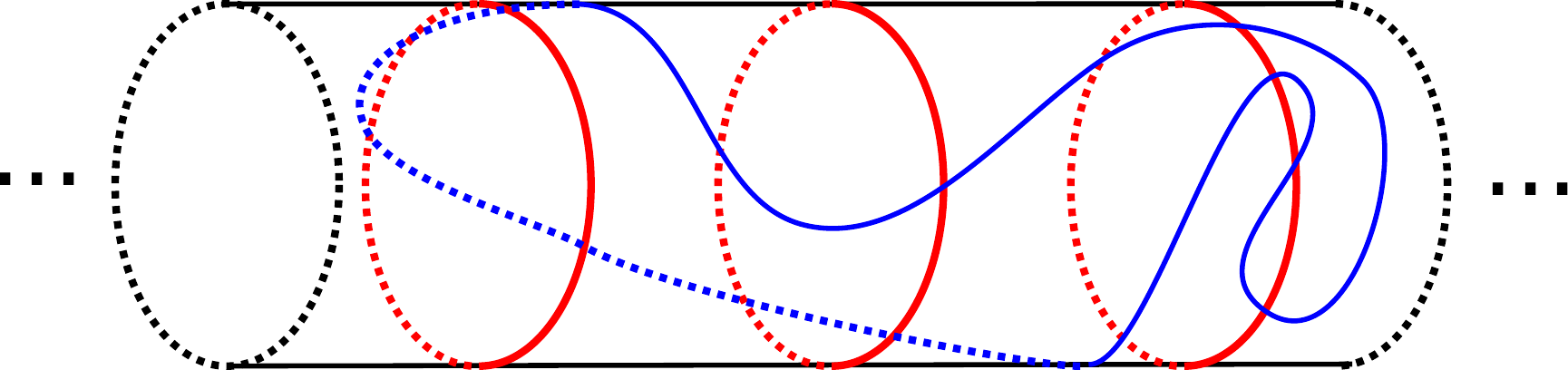}
          \caption{Crossing number via annuli}
          \label{fig:crossing-number}
        \end{figure}
	
	When $\alpha$ and $\beta$ are homotopic essential simple closed curves, there is a useful
        alternative way to describe $C_\alpha(\beta)$. Namely, let
        $A$ be the annular covering of $T$ corresponding
        to the cyclic subgroup generated by $\alpha$. We can identify $A$
        with $S^1 \times \RR$ so that the lifts $\hat{\alpha}_m$ of $\alpha$ are 
        the circles $S^1 \times \{m\}, m \in \ZZ$.

Then the crossing number $C_\alpha{\beta}$ is equal to the cardinality of  \[ \{ \hat{\alpha}_ m :    \hat{\beta}\cap\hat{\alpha}_m \neq \emptyset \},\] where $\hat\beta$ is an elevation of $\beta$. In other words, the number of lifts of $\alpha$ which 
  $\hat{\beta}$ intersects (compare
        Figure~\ref{fig:crossing-number}). Observe that this last
        characterisation is independent of the identification of the
        annulus with $S^1\times \RR$.
	
	\begin{lemma} \label{lem:crossing-estimate} 
          If $\alpha$ and $\beta$ are isotopic simple closed curves on
          $T$, then 
          \[C_\alpha(\beta)+1 \geq d^\dagger(\alpha, \beta).\]
	\end{lemma} 
	
	\begin{proof}
          We will perform a surgery replacing $\alpha$ with a
          homotopic curve $\alpha'$ such that $i(\alpha, \alpha') = 0$
          and $C_{\alpha'}(\beta) \leq C_\alpha(\beta) - 1$. By
          induction, this is enough to prove the lemma.
		
		To this end, we fix the cover $A$ and a lift $\hat{\beta}$ as in the discussion before the lemma, and keep the notation  $\hat{\alpha}_m = S^1 \times \{m\}$ for the lifts of $\alpha$.  
		Choose $a < b \in \mathbb{Z}$ with $b-a$ minimal, so that $\hat{\beta} \subset S^1 \times [a,b]$.  Note that, if $\hat{\beta}$ intersects $\hat\alpha_b = S^1 \times \{b\}$, then we may replace $\alpha$ with a nearby parallel copy of itself $\alpha'$ avoiding these finitely many points or intervals of intersection, and decreasing crossing number, thus already satisfying our desired outcome.
		Thus, we may assume $\hat{\beta}$ does not intersect $S^1 \times \{b\}$.  
		Consider the intersection of $\hat{\beta}$ with $S^1\times[b-1,\infty)$. This is a collection of arcs
		$b_1, \ldots, b_k$
		with each arc intersecting $\hat{\alpha}_{b-1}$ only in its endpoints. 
		Choose disjoint, closed disks $B_1, \ldots, B_r$ with the following properties:
		\begin{enumerate}
			\item Each $B_i$ is bounded by a segment contained in $\hat{\alpha}_{b-1}$, and one of the $b_j$.
			\item Each $b_j$ is contained in one of the $B_i$.
			\item All $B_i$ are disjoint from $\hat{\alpha}_i, i\neq b-1$.
		\end{enumerate} 
		By the third property, the $B_i$ map to disjoint, embedded disks in $T$ under the covering map.
		         \begin{figure}
	 \labellist 
  \hair 2pt
\pinlabel $\hat{\alpha}_{b-1}$ at 30 -10
\pinlabel \textcolor{red}{$\hat{\alpha}'$} at 140 60
\pinlabel $B_1$ at 110 100
\pinlabel $B_2$ at 80 30

   \endlabellist

          \centering
          \includegraphics[width=0.4\textwidth]{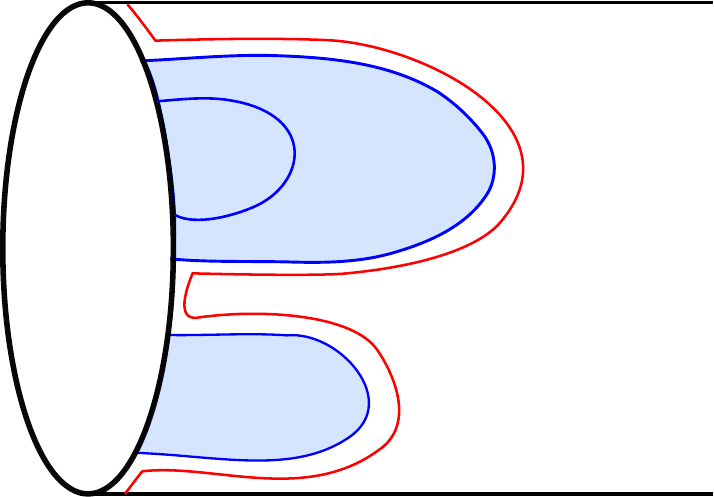}
          \caption{Surgery of $\alpha$ along arcs of $\beta$}
          \label{fig:surgery}
        \end{figure}
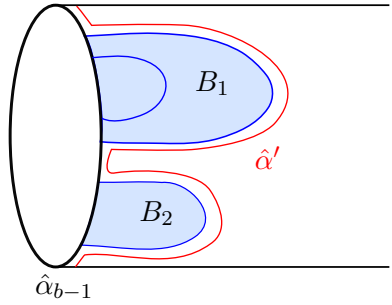

		We let $\alpha'$ be the curve obtained from surgery of
                $\alpha$ at all $b_j$ which appear as boundary
                segments of $B_i$.  Since this surgery can be done on
                $A$, $\alpha'$ is homotopic to $\alpha$. Since all
                $b_j$ approach $\hat{\alpha}_{b-1}$ from the same side, the surgered curve $\alpha'$
                can be chosen to be disjoint from $\alpha$. Hence, the
                cover defined by $\alpha'$ is still $A$. The lifts
                $\hat{\alpha}'_i$ are obtained from $\hat{\alpha}_i$
                by surgering at deck group translates of the $B_i$.
		
		Thus, the lift $\hat{\beta}$ can only intersect $\hat{\alpha}'_i$ if it also intersects $\hat{\alpha}_i$.
		By property 2, $\hat{\beta}$ does not intersect $\hat{\alpha}'_{b-1}$. Thus, we have $C_{\alpha'}(\beta) \leq C_\alpha(\beta) - 1$ by the description of crossing number before the lemma.
	\end{proof} 
	
	\begin{remark}
		One could define a version of crossing number also for surfaces of higher genus by using the Bass--Serre tree of the cyclic splitting of the fundamental group determined by $\alpha$. In this case the diameter of the projection to the Bass--Serre tree will provide an upper bound on distance, the idea of the proof is the same.   However, since we do not need this for our intended applications, we do not pursue this here.   
		
	\end{remark}
	
	\section{Hyperbolic Elements} \label{sec:hyperbolic} 
	In this section, we will study homeomorphisms isotopic to the identity acting
        hyperbolically on $\cd$ and prove Theorem~\ref{thm:hyperbolic-characterisation} from the introduction.

\subsection{Proof of Theorem~\ref{thm:hyperbolic-characterisation}}

Recall that Theorem~\ref{thm:hyperbolic-characterisation} asserts the equivalence of the following statements for $f \in \Homeo_0(T)$:
       
	\begin{enumerate}
	\item $f$ acts hyperbolically on $\cd(T)$,
	\item the rotation set $\rho(f)$ of $f$ has non-empty interior, and
	\item there is a finite $f$-invariant set $P \subset T$ such that $f$ represents a pseudo-Anosov mapping class of $T - P$.  
	\end{enumerate}
	For the proof, we will need the following consequence of Lemma~\ref{lem:crossing-estimate}.  
	\begin{lemma} \label{cor:bound} 
		If $f$ acts hyperbolically on $\cd(T)$ then for any curve $\alpha\in\cd(T)$ and for all $n>0$
 we have
		\[C_\alpha (f^n(\alpha)) \geq |f|\cdot n - 1.\] 
	\end{lemma}	
	\begin{proof} 
          
	  Using Lemma~\ref{lem:crossing-estimate} we
          have for all $n>0$ \[|f|\cdot n \leq d^\dagger(\alpha, f^n(\alpha)) \leq
          C_\alpha(f^n(\alpha))+1.\qedhere \] 

	\end{proof}

	\begin{proof}[Proof of Theorem~\ref{thm:hyperbolic-characterisation}]
			The assertion $(2) \Rightarrow (3)$ is Theorem~\ref{thm:LM} of Llibre--MacKay \cite{LM},
			and $(3) \Rightarrow (1)$ is Theorem~\ref{thm:TL}. Thus, we need only show the implication $(1) \Rightarrow (2)$.  
		
			Suppose $f$ acts hyperbolically on $\cd(T)$. Identify $T$ with
			$\RR^2/\ZZ^2$, and let $\alpha$ be the simple closed curve whose lifts to $\RR^2$ are the horizontal lines in $\RR^2$ with integer second co-ordinates. Let $\widetilde{f}$ be a lift of $f$ to $\RR^2$. 
			
			By Lemma~\ref{cor:bound} we have \[C_\alpha(f^n(\alpha)) \geq |f|\cdot n - 1.\]  
			This means that for any $n$, lifts of $f^n(\alpha)$ intersect at least $|f|\cdot n - 1$ distinct
			horizontal lines (with integer second co-ordinate) in the plane $\RR^2$. In other words, there are sequences of points $z_n$ and $z'_n$
			(on the same horizontal line) so that, for each $n\in\NN$ the points $\widetilde{f}^n(z_n)$ and $\widetilde{f}^n(z'_n)$ have $y$-coordinate differing by at least $|f|\cdot n - 1$, hence the displacements $d_n \coloneq  f^n(z_n) - z_n$ and $d'_n \coloneq f^n(z'_n) - z'_n$ have $y$-coordinates differing by at least $|f|\cdot n - 1$.  This implies that the projection of $\rho(f)$  to the $y$-axis has diameter at least $|f|$. 
			
Since rotation sets are convex and compact, we need only now rule out the possibility that $\rho(f)$ is a line segment.   Suppose for contradiction that $\rho(f)$ were a line segment, $\ell$.    If $\ell$ has rational slope, then we may find $A \in \SL_2(\ZZ)$ such that $A(\ell)$ is a subset of the $x$-axis.  As discussed in Section~\ref{sec:rot}, we have $\rho(AfA^{-1}) = A\rho(f)$. Since asymptotic translation length is a conjugacy invariant, we have $|AfA^{-1}| = |f|$, so we may apply the same argument as above to conclude that projection of $\rho(AfA^{-1}) = A \rho(f)$ to the $y$-axis has diameter at least $|f|$, a contradiction.    If the slope of $\ell$ is irrational, we can again find $A \in \SL_2(\ZZ)$ such that $A(\ell)$ has projection to the $y$-axis a set of diameter less than $|f|$: in fact, with a process similar to Euclid's algorithm, one may apply elementary matrices to make the projection of the length of the segment $\ell$ arbitrarily small, again contradicting the lower bound. This concludes the proof.  
	\end{proof} 

In the following proposition we bound the area of the rotation set from below by a constant multiple of the square of the asymptotic translation length on $\cd(S)$. There is no analogous upper bound because there are rotation sets with arbitrarily large area but of bounded height, and height bounds the translation length from above by Lemma~\ref{cor:bound}.

\begin{proposition}\label{prop:areabound}
Let $f\in\Homeo_0(T)$ act hyperbolically on $\cd(T)$. Then \[ \mathrm{Area}\rho(f) \geq \frac{\sqrt{3}}{8} |f|^2.\]\end{proposition}

\begin{proof} First note that $\rho(f)$ is compact, and has non-empty interior by Theorem~\ref{thm:hyperbolic-characterisation}, so we may pick a maximal area parallelogram $P\subset \rho(f)$. Now find $A\in \SL_2(\RR)$ such that the diagonals of the parallelogram $A\cdot P$ are horizontal/vertical, and of the same length. Thus $A\cdot P$ is equal to a closed $r$-ball in the $L^1$ norm centered at the intersection of the diagonals $Q$. Since $A\cdot \rho(f)$ is also convex, and $A\cdot P\subset A \cdot\rho(f)$ has maximal area, we must have that $A\cdot \rho(f)$ is contained within the closed $r$-ball in the $L^\infty$ norm with the same centre $Q$. This in turn is contained in a disk $D$ of radius $\sqrt{2}r$. Thus $\mathrm{Area} D = \pi\mathrm{Area} P.$ It suffices to bound the diameter of $D$ from below to establish a lower bound on $\mathrm{Area}\rho(f)$.

After post composing by a rotation, we may assume that $(h,0)\in A \cdot \ZZ^2$ is a non-zero vector of smallest length and $h>0$. It is well known that $h^2 \leq \frac{2}{\sqrt{3}}$. Let $(k,v)\in A \cdot \ZZ^2$ be a smallest length vector outside the span of $(h,0)$. Then we must have $v=\frac{1}{h}$ because $A$ is area preserving. Now $D$ contains $A\cdot \rho(f)$, so by Lemma~\ref{cor:bound} the projection of $D$ to the second co-ordinate is at least $\frac{1}{h} \cdot |f|$. This gives the desired lower bound on $\mathrm{Area}D=\pi \mathrm{Area}P$ and hence $\mathrm{Area}\rho(f)$ as required.\end{proof}

\subsection{General characterisation of hyperbolic isometries} \label{sec:rot_general}

In this section we prove the following extensions of Theorem \ref{thm:hyperbolic-characterisation}. 

%%%%%%%%%. COPY PASTING of Statement that used to be in introduction %%%%%%
\begin{theorem}[Characterisation of hyperbolic homeomorphisms] \label{thm:hyperbolic-characterisation-general}
Let $f$ be an orientation preserving homeomorphism of $T$. Then $f$ acts hyperbolically on $\cd(T)$ if and only if $f$ satisfies one of the following (mutually exclusive) conditions.
\begin{enumerate}
\item The homeomorphism $f$ is isotopic to an Anosov homeomorphism of $T$,
\item a finite power of $f$ is isotopic to a Dehn twist map about some simple closed curve $\alpha$, and its rotation set $\rho_\alpha$ has non-empty interior, or
\item a finite power of $f$ is isotopic to the identity and its rotation set has non-empty interior.
\end{enumerate}
\end{theorem}
Of course, if a homeomorphism $f$ of $T$ reverses the orientation, $f$ acts hyperbolically on $\cd(T)$ if and only if the orientation-preserving homeomorphism $f^2$ does so. We hence have a complete characterisation of elements of $\Homeo(T)$ which act hyperbolically on $\cd(T)$.

\begin{proof} 
For the first statement, it is a special (and easy) case of the Nielsen--Thurston classification theorem that any orientation-preserving homeomorphism of $T$ either is isotopic to an Anosov homeomorphism of $T$, or has a power isotopic to a Dehn twist map of $T$, or has a power isotopic to the identity.   In the first case, it acts hyperbolically on $\cd(T)$ since it acts hyperbolically on the Farey graph $\C(T)$, and there is a $1$-coarsely Lipschitz map $\cd(T)\to \C(T)$ via sending $\alpha$ to its isotopy class $[\alpha]$. In the third case, Theorem~\ref{thm:hyperbolic-characterisation} states that $f$ acts hyperbolically if and only if its rotation set has nonempty interior.  Thus, we need only understand the case where $f$ is isotopic to a Dehn twist. 

To this end, suppose $f$ is isotopic to a Dehn twist around the essential simple closed curve $\alpha$. Assume first that $\rho_\alpha(f)$ has non-empty interior.  Then, by a theorem by Doeff (see \cite[Theorem~6.5]{Doeff}), $f$ is isotopic to a pseudo-Anosov homeomorphism relative to a finite subset $P$ of $T$. By Theorem~\ref{thm:TL}, this implies that $f$ acts hyperbolically on $\cd(T)$.

Now, suppose that $\rho_\alpha(f)$ has empty interior, in which case it is a singleton. With the same methods as for homeomorphisms isotopic to the identity, we will prove that $f$ cannot act hyperbolically on $\cd(T)$.  Denote by $\check{T}$ the cyclic covering associated to $\alpha$. Take any lift $\check{\alpha}$ of the loop $\alpha$ to $\check{T}$ and fix a lift $\check{f}\colon\check{T} \rightarrow \check{T}$ of $f$. We identify $\check{T}$ with $\mathbb{R}/\mathbb{Z} \times \mathbb{R}$, where the group of deck transformations of the covering map $\check{T} \rightarrow T$ is the group of integral translations of $\mathbb{R}/\mathbb{Z} \times \mathbb{R}$. Let $p_{2}\colon \check{T}=\mathbb{R}/\mathbb{Z} \times \mathbb{R} \rightarrow \mathbb{R}$ be the projection. For any $n \geq 0$, denote by $D_n$ the diameter of $p_{2}(\check{f}^{n}(\check{\alpha}))$ and observe that $\lfloor D_{n} \rfloor \leq C_\alpha(f^{n}(\alpha)) \leq \lfloor D_n \rfloor+1$.

As the rotation set of $f$ is a singleton, 
$$ \lim_{n \rightarrow +\infty} \frac{D_n}{n}=0$$
so 
$$ \lim_{n \rightarrow +\infty} \frac{C_\alpha(f^{n}(\alpha))}{n}=0.$$
Hence, by Lemma~\ref{lem:crossing-estimate}, 
$$ \lim_{n \rightarrow +\infty} \frac{d^{\dagger}(\alpha,f^{n}(\alpha))}{n}=0$$
and $f$ does not act hyperbolically, concluding the proof of Theorem \ref{thm:hyperbolic-characterisation-general}.
\end{proof}

	\section{Parabolic and Elliptic Elements} \label{sec:parabolic}
	Given Theorem~\ref{thm:hyperbolic-characterisation}, one might na\"ively hope for a similar characterisation of elliptic and parabolic isometries of $\cd(T)$ in terms of their rotation sets.    In this section we show this na\"ive hope is too optimistic, giving sufficient dynamical criteria for elliptic and parabolic actions on the torus, then showing why these cannot be necessary.   Finally, we return from the torus to surfaces of higher genus and complete the proof of Theorem \ref{thm:parabolicsexist}.
	
\subsection{Sufficient dynamical criteria}
Recall Theorem \ref{irrationalimpliesparabolic} is the claim that, for $f \in \Homeo_0(T)$, 
\begin{quote}
\begin{enumerate}
\item if $\rho(f)$ is an irrational slope segment, then $f$ acts on $\cd(T)$ as a parabolic isometry, and 
\item if $\rho(f)$ is a rational slope segment through a rational point, then the action of $f$ is elliptic.
\end{enumerate} 
\end{quote} 

\begin{proof}[Proof of Theorem \ref{irrationalimpliesparabolic}, rational slope through rational point case]
This case follows from work of D\'avalos \cite{Davalos} and our crossing number estimates.   Suppose that $f$ is isotopic to the identity and the rotation set $\rho(f)$ is a rational slope segment through a rational point.
By Theorem~A in \cite{Davalos}, $f$ is \emph{annular}, 
		which means that there is some $v\in \ZZ^2$ so that for any lift $F$ of $f$ to $\RR^2$ we have
		\[ |\langle F^n(x)-x, v \rangle| \leq M \]
		for some $M$ and all $x \in \RR^2$. This implies that the crossing numbers $C_\alpha(f^n\alpha)$ (as in Section~\ref{sec:estimates})
		are bounded for suitable $\alpha$, and thus $f$ acts elliptically, as desired.  
\end{proof} 

\begin{remark} 
A similar argument can be applied to homeomorphisms isotopic to a Dehn twist map around an essential simple closed curve $\alpha$ whose rotation set consists of a single rational point. Denote by $\check{T}$ the cyclic cover associated to $\alpha$, by $\check{f}\colon\check{T} \rightarrow \check{T}$ a lift of $f$ to this cyclic cover and, by $t$ a positive generator of the group of deck transformations of $\check{T} \rightarrow T$.  Then, if the rotation set of $\check{f}$ is $\left\{ \frac{p}{q} \right\}$, by a theorem by Addas-Zanata, Garcia and Tal (see \cite[Theorem 2]{AGT}), the homeomorphism $t^{-p}\check{f}^{q}$ has a compact invariant set which separates $\check{T}$. Hence $f$ acts elliptically on $\cd(T)$.
\end{remark} 
	
To treat the case of irrational slope, we use the following construction. 

\begin{construction} \label{const:cover}
For $n \in \NN$, let $T_n \to T$ be the cover of the torus defined by reducing mod $n$ in homology:
	\[ \pi_1(T) \to (\ZZ/n\ZZ)^2 \]
	We equip $T$ with the usual flat metric inherited
	from $T = \RR^2 / \ZZ^2$, and we equip $T_n$ with
	the metric obtained by pulling back the metric of $T$ to $T_n$, and rescaling by $1/n$. Observe that
	$T$ and $T_n$ are isometric, but not via the covering map.
\end{construction} 
	
The following lemma says that this covering behaves nicely with respect to the associated curve graphs:
	\begin{lemma}\label{lem:unwrapn}
		Suppose that $\alpha$ and $\beta$ are adjacent in $\cd(T)$. Then any choice of lifts of $\alpha$ and $\beta$ in $T_n$ are adjacent in $\cd(T_n)$.
	\end{lemma}	
	\begin{proof}
		Let $\alpha'$ and $\beta'$ be any choice of lifts in $T_n$ of $\alpha$ and $\beta$. If $\alpha$ and $\beta$ were disjoint then it is immediate that $\alpha'$ and $\beta'$ are disjoint. On the other hand if $|\alpha\cap\beta|=1$ then, after applying a homeomorphism of $T$ (which lifts to the characteristic cover $T_n$), we can suppose that $\alpha$ and $\beta$ are the standard generators of the fundamental group of $T=\mathbb{R}^2/\mathbb{Z}^2$. 
Then it is immediate that lifts $\alpha$ and $\beta$ to $T_n$ have only one intersection point and are thus adjacent in $\cd(T_n)$.  
	\end{proof}

        We also need the following consequence of
        Lemma~\ref{lem:small-diameter-one-cover}.
	\begin{lemma} \label{lem:goodfinitecover}
          Given any $K\geq 0$ there exists a translation surface $X=X(K)$ with the
          following property:

          For any $n$ and any $p \in T_n$, there is a 
          cover $f\colon X \to T_n$ branched only over $p$ so that:
          \begin{enumerate}
          \item If $\alpha$ and $\beta$ are curves on $T_n$ disjoint
            from $p$ with
            \[ d^\dagger(\alpha,\beta)\leq K+1, \] 
            then $\alpha, \beta$ admit disjoint elevations in the branched cover $f$.
          \item The branched cover $f$ is a local isometry at each nonsingular point.
          \end{enumerate}
	\end{lemma}
	We emphasise that in the lemma, the \emph{covering map} $f$ may depend
        on $p$, but the geometry of the translation surface $X$ does not.
        In addition, the geometry of $X$ does not depend on $n$.

	\begin{proof} 
First, we show
          that given $K$ and $p \in T_n$, there is a translation surface
          $X(K,p)$ which has the desired properties for $T_n$, we then show independence from $n$ and $p$.  

          To prove this, begin by observing that there is a
          $3$-fold branched cover $f_1\colon Y\to T_n$, branched only at $p$,
          where $Y$ has genus $2$.
          Suppose that $d^\dagger(\alpha,\beta)\leq K+1$, and let
          $\alpha'$ and $\beta'$ be any choice of elevations of
          $\alpha$ and $\beta$ to $Y$. Observe that
          $d^\dagger(\alpha',\beta')\leq 2K+2$ in $\cd(Y)$. This is
          because any choice of elevations of adjacent curves in $T_n$
          to $Y$ will intersect at most three times, and therefore
          have distance at most $2$ in $\cd(Y)$. Now we may apply
          Lemma~\ref{lem:small-diameter-one-cover} to find a
          finite-sheeted cover $f_2\colon X \to Y$, which only depends on
          $K$ (and $Y$), such that any $\alpha'$ and $\beta'$ as above
          admit disjoint elevations to $X$. The branched cover $f_p =
          f_1\circ f_2\colon  X(K,p)\to T_n$ then has (1).  By pulling back
          the translation surface structure from $T_n$ to $X(K,p)$, we
          can also satisfy (2).

          Now if $p'$ is any other point, observe that there is an
          isometry $\iota\colon T_n \to T_n$ mapping $p$ to $p'$. Then
          $f_{p'} = \iota \circ f_1\circ f_2\colon X(K,p)\to T_n$ has the
          desired property (1) for curves disjoint from the point
          $p'$. 
          Since $\iota$ is an isometry, $f_{p'}$ also satisfies
          (2). Hence, the translation surface $X(K,p)$ can be chosen
          not to depend on $p$. 
          Since all $T_n$ are isometric, the surface $X(K)$ can also be 
          chosen to be the same for all $T_n$.
	\end{proof}

Finally, we use the following observation, likely well known to experts.  We include a short proof for completeness. 
	\begin{lemma} \label{lem:dense} Let $X$ be a compact square-tiled translation surface and $\lambda \notin \QQ$. Then there exists $L>0$ such that any straight line of length $L$ and slope $\lambda$ that is disjoint from the singularities of $X$ will intersect all horizontal curves on $X$.
	\end{lemma}
	
\begin{proof}
Since $\lambda$ is irrational, every half-leaf of the foliation with slope $\lambda$ is dense in $X$.   Supposing for contradiction that we can find a sequence of segments $L_n$ of length $\ell_{n} \to \infty$ so that each $L_n$ misses some horizontal curve $c_n$. We parametrize the segments $L_n$ by arclength so that we see the segments $L_{n}$ as continuous maps $[-\ell_{n}/2, \ell_{n}/2] \rightarrow X$. After possibly taking a subsequence, we can assume that these $L_n$ converge uniformly on any segment of $\mathbb{R}$ to an infinite leaf $\mu$ of slope $\lambda$. 

Since  $\mu$ is dense and of constant slope, some compact subinterval $\mu' \subset \mu$ will intersect all horizontal curves, and we may find such a $\mu'$ which does not contain any singularity of $X$. Thus, by compactness, there exists $\epsilon>0$ such that the $\epsilon$-neighborhood $N$ of $\mu'$ is disjoint from the singularities of $X$. Since the $L_n$ converge to $\mu$ locally on compact subintervals, the intersection $L_n\cap N$ is non-empty for sufficiently large $n$. Now any connected component of $L_n\cap N$ which does not contain the endpoints of $L_n$ 
 is also a line of slope $\lambda$, and must intersect all horizontal curves because $\mu'$ can be isotoped along horizontal lines to a subset of $L_n \cap N$. But this contradicts the definition of $L_n$, completing the proof.
\end{proof}

\begin{proof}[Proof of Theorem~\ref{irrationalimpliesparabolic}, irrational slope case]

		Suppose $\rho(f)$ is a line segment of irrational slope $\lambda$ and length $\ell>0$. 
		By Theorem~\ref{thm:hyperbolic-characterisation}, the homeomorphism $f$ is not hyperbolic, as the rotation set has empty interior. Thus, to prove the theorem, it suffices to show $f$ is not elliptic. Let $\alpha$ be the standard horizontal curve on $T$. Supposing for contradiction that $f$ is elliptic, there exists some $K$ such that $d^\dagger(\alpha, f^n \alpha)\leq K$ for every $n$.
			
		Consider the finite-sheeted covers $T_n \to T$ and metrics as described above.  Let $D = [0, 1] \times [0,1]$ denote the 
		standard fundamental domain for $T=\RR^2/\ZZ^2$. We also view $T_n$ as the standard quotient $\RR^2 / \ZZ^2$ because $T_n$ is isometric to $T$. We choose $D_n$ to be the subset $[0,1/n]\times [0,1/n]$ in $\RR^2$.  Fix a lift $\tilde{f}\colon \RR^2 \to \RR^2$ of $f\colon T \to T$. Now we may also construct $\tilde{f}_n \colon \RR^2 \to \RR^2$ defined by \[\tilde{f}_n(x)=\frac{1}{n} \tilde{f}(nx),\]
which descends to a map $\hat{f}_n \colon T_n \to T_n$, which is a lift of $f$.

Let $L> 0$ be the length guaranteed by Lemma~\ref{lem:dense} applied to the
surface $X=X(K)$ guaranteed by Lemma~\ref{lem:goodfinitecover} and slope $\lambda$, i.e.
any straight line on $X$ of slope $\lambda$ and length $L$ must intersect every horizontal curve on $X$.
		
		Now fix some $k \in \NN$ such that $k\ell >L$.  
By Lemma \ref{lem:MZ} we have that $\frac{1}{n} \tilde{f}^n (D)$ converges to $\rho(f)$ in the Hausdorff topology as $n \to \infty$. 
		 By definition we have that $\tilde{f}_n^{kn} (D_n)=\frac{1}{n} \tilde{f}^{kn} (D)$, and hence $\tilde{f}_n^{kn} (D_n)$ also converges to $k\rho(f)$. 
Let $\overline{k\rho}_n \subset T_n$ denote the projection of $k\rho(f)$, where we drop $f$ from the notation for convenience.

		We claim now that for each $n$ there is a straight line $a_n$ in $T$ (disjoint from the standard horizontal curve $\alpha$), with lift $\hat{a}_n$ in $T_n$, such that the Hausdorff distance between $\hat{f}_n^{kn}(\hat{a}_n)$ and $\overline{k\rho}_n$ tends to $0$ as $n$ tends to $\infty$. To find such a straight line, pick two points $q_{n}$ and $q'_{n}$ in  ${\hat{f}}_n^{kn}(D_n)$ which are close to the two ends of $\overline{k\rho}_n$. Then take as $\hat{a}_n $ the straight line of $D_{n}$ which joins the points $\hat{f}^{-kn}(q_{n})$ and $\hat{f}^{-kn}(q'_{n})$, and $a_n$ its projection to $T$. 
		
		Now define simple closed curves $\alpha_n$ on $T$ such that $\alpha_n$ contains $a_n$ but is also isotopic to and disjoint from the standard horizontal curve $\alpha$. 
		For large enough $n$ we have that the Hausdorff distance between $\hat{f}_n^{kn}(D_n)$ and $\overline{k\rho}_n$ on $T_n$ tends to $0$ as $n$ tends to $\infty$. Write $\hat{\alpha}_n$ for the lift of $\alpha_n$ to $T_n$ that contains $\hat{a}_n$, and let $\hat\alpha$ be an arbitrary lift of $\alpha$ to $T_n$.

		By our initial assumption on $f$ we have $d^\dagger(\alpha,f^{kn}\alpha_n)\leq K+1$ for every $n>0$. By Lemma~\ref{lem:unwrapn} we also have $d^\dagger(\hat\alpha,\hat{f}_n^{kn} \hat\alpha_n)\leq K+1$ for all $n>0$ in $\cd(T_n)$. For any
point $p$ disjoint from both $\hat\alpha$ and $\hat{f}_n^{kn} \hat\alpha_n$, applying Lemma~\ref{lem:goodfinitecover}, we find that there are branched covers $\pi_n\colon X \to T_n$, so that
$\hat\alpha$ and $\hat{f}_n^{kn}\hat\alpha_n$ admit disjoint lifts to $X$.

By our choices, there is a lift of $\hat{f}_n^{kn}\hat{a}_n$ to $X$ which is disjoint from some horizontal curve in $X$ for each $n$. But $\hat{f}_n^{kn}\hat{a}_n$ converges to $k\rho(f)$, which has slope $\lambda$ and length $k\ell >L$.
By choosing the branch point $p$ outside a small embedded regular neighbourhood of $k\rho(f)$, we can guarantee that lifts of  $\hat{f}_n^{kn}\hat{a}_n$ to $X$ contain segments which converge to a segment with slope $\lambda$ and length $k\ell >L$ as well. 
For sufficiently large $n$ this contradicts Lemma~\ref{lem:dense}.
			\end{proof}
	
\subsection{Parabolics with a singleton rotation set}  \label{sec:para_examples}
	
	It is easy to produce examples of elliptic isometries with a given singleton rotation vector.  Given $(a, b)$, the translation $(x, y) \mapsto (x+a, y+b)$ is a lift of a torus map with rotation set $\{ (a,b) \}$.   There are many dynamically more interesting examples as well. For instance, Koropecki and Tal constructed a homeomorphism of the torus whose rotation set is reduced to $\{ (0,0) \}$ but which has unbounded orbits in every direction in the universal cover (see \cite{KT}). Similar examples can be constructed with any rational one-point rotation set. We believe that their example acts as a parabolic isometry of $\cd(S)$.	
	
	In this section we give a construction that  produces parabolic isometries with singleton rotation sets. 
	\begin{proposition}\label{prop:parabolic-with-point-rot}
   There are homeomorphisms $f$ that act parabolically on $\cd(T)$ with 
           $\rho(f) = \{(0,0)\}$. 
	\end{proposition}
	As we remark later, a small modification of the construction can be used to produce singleton rotation sets other than $\{(0,0)\}$.  We first give the construction, then prove it is parabolic.  

\begin{construction}  \label{const:parabolic}
	Start with a standard Denjoy counterexample map
        $D \colon S^1 \to S^1$ with irrational rotation number $\alpha$, produced by blowing up a single orbit of a standard irrational rotation. 
        See \cite{Ghys_circle}
        for an introduction to these examples.  
        We interpret $T = S^1 \times S^1$ as the mapping
        torus of $D$, and let $\varphi_t$ be the suspension
        flow. The time-$1$--map $\varphi_1$ of this flow preserves the
        foliation of $T$ into circles, acting with rotation number
        $\alpha$ on each of them. 
        We think of the circles as being
        \emph{horizontal}, and the flow lines as being
        \emph{vertical}.
	
	Let $K$ be the suspension of the minimal invariant Cantor set
        of $D$; the set $K$ is invariant under the flow $\varphi_t$. Denote
        by $\lambda$ the suspension of the boundary points.
        Now let $J$ denote a closed interval in a horizontal circle whose interior $\mathring{J}$ is a wandering interval for the Denjoy circle map. Then we obtain an embedding $\Phi \colon U = \mathring{J} \times \mathbb{R} \to T$ given by flowing this transverse segment, whose complement is exactly the suspension of the invariant Cantor set of $D$. We denote the coordinates given by this map by $(x,t)$. We now define a map preserving flow lines of $U$ via the map
        $$(x,t) \longmapsto (x, t + \eta(x)e^{-|t|} )$$
        where  $\eta(x) <1$ is a bump function supported on $J$ that is strictly positive on the interior. This map extends to continuously over $K$ as the identity. We call the resulting map $f$. We will prove that it is the desired counterexample. Note that $f$ preserves flow lines of $\varphi_t$, moving along lines montonically and with exponentially decreasing speed as one approaches the set $K$. 
        \end{construction}

	Proposition~\ref{prop:parabolic-with-point-rot} will be a
        consequence of the following two claims:
	\begin{claim}\label{claim:rotset}
          The rotation set of $f$ is $\{(0,0)\}$.
       
	\end{claim}
	\begin{claim}\label{claim:para}
          The homeomorphism $f$ acts parabolically on $\cd(T)$.
	\end{claim}

	\begin{proof}[Proof of Claim~\ref{claim:rotset}]
          Denote by $\widetilde{\varphi}_t,\widetilde{f}$ respective lifts of $\varphi_{t},f$ to
          the universal cover $\RR^2$ of $T$.  For any $p= (\tilde{x},t)  \in \RR^2$ in (a lift of) $U$ we have using
          growth condition ($*$) above that
          
          \[ \lim_{i \to \infty} \frac{\widetilde{f}^{n_i}(p)-p}{n_i} <  \lim_{i \to \infty} \frac{\widetilde{f}^{n_i}(p_k)-p}{n_i} = \lim_{i \to \infty} \frac{\widetilde{f}^{n_i}(p_k)-p_k}{n_i}  < e^{-t_k} \to 0, \]
         where $p_k = \widetilde{f}^{k}(p) = (\tilde{x},t_k)$ denotes an iterate of $p$ so that $t_k \to \infty$ and the claim follows.
	\end{proof}

	\begin{proof}[Proof of Claim~\ref{claim:para}] 
          Given Claim~\ref{claim:rotset} and
          Theorem~\ref{thm:hyperbolic-characterisation}, it follows
          that the homeomorphism is not hyperbolic. 
          
          We thus have to exclude the case that it is elliptic.  To this end, suppose
          that it were. Let $\gamma$ be a horizontal curve on $T$
          containing the segment $J$ from property (4). By hypothesis, the sequence 
          $( d^\dagger(\gamma , f^n(\gamma)))_{n} $ is bounded.
          
         By Lemma~\ref{lem:goodfinitecover}, there exists a translation surface $X$ and a finite cover $ X \to T$ of bounded order, branched at a single point $p$ (independent of $n$), so that, for any $n$, the curves $\gamma, f^n(\gamma)$ admit disjoint lifts $\tilde{\gamma}$ and $\tilde{\beta}_{n}$ to $X$ and the cover is locally isometric. Note that we are free to pick $p$ outside $\gamma$ and $J \times \mathbb{R}$.
		
          By Lemma~\ref{lem:dense}, there is a number $L>0$ so that any segment of slope $\alpha$ and length greater than $L$ of $X$ intersects every horizontal curve. 
          
          Claim~\ref{claim:para} will be a consequence of the following technical claim.
          
          \begin{claim}\label{claim:techn}
          There exists a sequence $k_{n} \rightarrow +\infty$ of integers and a sequence $(L_{n})$ of straight segments of $T$ of slope $\alpha$ such that
          \begin{enumerate}
          \item for any $n \geq 0$, there exists a subsegment $\sigma_{n}$ of $f^{k_{n}}(J \times \left\{ 0 \right\})$ such that $d_{Hausdorff}(L_{n},\sigma_{n}) 
          \rightarrow 0$. 
          \item the length of $L_{n}$ tends to $+\infty$.
          \item for any $n \geq 0$, the segment $L_{n}$ does not meet the singularity $p$. 
          \end{enumerate}
          \end{claim}
          
          Before proving this claim, let us use it to prove Claim \ref{claim:para}. For any $n \geq 0$, let $\tilde{\sigma}_{n}$ be the lift of $\sigma_{n}$ to $X$ which is contained in $\tilde{\beta}_{n}$, $\tilde{L}_{n}$ be the lift of $L_{n}$ which is close to $\tilde{\sigma}_{n}$. Taking a subsequence if necessary, the sequence $\tilde{L}_{n}$ converges (on compact subsegments) 
 to a straight line $\tilde{L}_{\infty}$ of infinite length (like in the proof of Lemma \ref{lem:dense}). Take a subsegment $L'_{\infty}$ of $L_{\infty}$ which does not meet the singularities of the covering map and whose length is greater than $L$. Take $\epsilon >0$ such that the $\epsilon$-neighbourhood of $L'_{\infty}$ does not meet the singularities of the covering map and such that any connected set at Hausdorff distance at most $\epsilon$ from 
$L'_\infty$ meets any horizontal curve. 

For $n \geq 0$ sufficiently large, some connected subset $\tilde{\sigma}_{n}'$ of $\sigma_{n}$ is at Hausdorff distance at most $\epsilon$ from $L'_{\infty}$ and hence meets any horizontal curve. But this is not possible as, by definition of $X$, $\tilde{\sigma}_{n} \subset \tilde{\beta}_{n}$ is disjoint from the horizontal curve $\tilde{\gamma}$.
          	\end{proof}
	
	\begin{proof}[Proof of Claim \ref{claim:techn}]
	Fix $n \geq 0$. Let $C_{n}=[a_{n}, b_{n}]=[D^{n}(a_{0}),D^{n}(b_{0})]$ be the connected component of the complement in $\mathbb{S}^1 \times \left\{ 0 \right\}$ of the minimal Cantor set of $D$ which contains $J\times \left\{ n \right\}$. As the $C_{n}$'s are pairwise disjoint and $\mathbb{S}^1 \times \left\{ 0 \right\}$ has finite length, the diameter of $C_{n}$ tends to $0$ as $n \rightarrow +\infty$.

Let $\tilde{D}\colon \mathbb{R} \rightarrow \mathbb{R}$ be a lift of $D$ to the universal cover of $\mathbb{S}^1=\mathbb{R} / \mathbb{Z}$. Recall that the rotation number of $D$ is $\alpha$. Then it is standard that, for any $n$,
$$\tilde{D}^{n}(\tilde{a}_{0})=n\tilde{\alpha}+B_{n},$$
where $(B_n)$ is a bounded sequence and $\tilde{a}_0$ and $\tilde{\alpha}$ are respective lifts of $a_0$ and $\alpha$ to $\mathbb{R}$ (see \cite[Section 5]{Ghys_circle}). Take a sequence $(k'_{n})_{n}$ of integers such that the sequence $(B_{k'_{n}})$ converges and $k'_{n+1}-k'_{n} \rightarrow +\infty$.  Observe that the sequence $(D^{k'_{n+1}}(a_{0})-D^{k'_{n}}(a_{0})-(k'_{n+1}-k'_{n})\alpha)$ converges to $0$. 
	
	Take $k_{n}$ sufficiently large so that the set $f^{k_{n}}(J\times \left\{ 0 \right\})$ meets $J \times [k'_{n+1},+\infty)$. Let $\tilde{\mathcal{L}}_{n}$ be the line of the universal cover $\mathbb{R}^2$ of $T$ which joins $(a_{k'_n},k'_{n})$ to $(a_{k'_{n+1}}, k'_{n+1})$ and $\tilde{L}_{n}$ be either the line which joins $(a_{k'_{n}},k'_{n})$ to $(a_{k'_{n}}+(k'_{n+1}-k'_{n}) \alpha,k'_{n+1})$ or a tiny translate of this line if it meets some lift of the point $p$. Denote by $\mathcal{L}_{n}$ and $L_{n}$ the respective projections of $\tilde{\mathcal{L}}_{n}$ and $\tilde{L}_{n}$ to $T$.
	
	As the sequence $(D^{k'_{n+1}}(a_{0})-D^{k'_{n}}(a_{0})-(k_{n+1}-k_{n})\alpha)$ converges to $0$, the Hausdorff distance between $\mathcal{L}_{n}$ and $L_{n}$ tends to $0$. As the set $f^{k_{n}}(J\times \left\{ 0 \right\})$ meets both $J \times \left\{ 0 \right\}$ and $J \times \left\{ k'_{n+1} \right\}$ and is contained in $J \times \mathbb{R}$, there exists a connected component $\sigma _{n}$ of $f  ^{k_{n}}(J \times \left\{ 0 \right\}) \cap J \times [k'_{n},k'_{n+1}]$ which meets both $J \times \left\{ k'_{n} \right\}$ and $J \times \left\{ k'_{n+1} \right\}$. As the diameter  of $C_{n}$ tends to $0$ as $n$ tends to $+\infty$, the Hausdorff distance between $\mathcal{L}_{n}$ and $\sigma_{n}$ tends to $0$. Finally, the Hausdorff distance between  $L_{n}$ and $\sigma_{n}$ tends to $0$.
	\end{proof}

One can alternatively show the map $f$ from Construction \ref{const:parabolic} is parabolic with the following argument using bicorns.  

\begin{proof}[Alternative proof of Claim \ref{claim:para}]
Recall 
we have identified $T$ with $[0,1] \times [0,1]$ with the top and bottom edges identified by the Denjoy map $D$, and the left and right by rigid translation.  
Let $\gamma = [0,1] \times \{0\}$ be the standard horizontal curve on $T$, let $J \subset \gamma$ be the interval used in the definition of $f$; we may parametrize $\gamma$ so that the left endpoint of $J$ is 0.  Choose a horizontal curve $\gamma' = [0,1] \times \{\epsilon'\}$ slightly above $\gamma$ so that $\gamma'$ is transverse to each curve $f^n(\gamma)$ (each iterate will have only countably many points tangent to a horizontal curve, so we may find such a $\gamma'$ as close as we wish to $\gamma$).  

Let $J' = [0, \epsilon] \subset J$ be a subinterval on which the bump function in the construction of $f$ is monotone increasing.  
For each $n>0$ large enough that $f^n(J')$ intersects $\gamma$ in at least two points, let $j_n \subset f^n(J')$ be the subarc bounded by the leftmost intersection point on $\gamma$ of $f^n(J') \cap \gamma'$  (which by construction lies in $[0,\epsilon] \times \{\epsilon'\}$), and the rightmost intersection point of $f^n(J') \cap \gamma'$ on $\gamma'$.  
 Let $c_n$ be the union of $j_n$ and the short segment of $\gamma'$ containing $0$ which connects them. This is a closed curve.  
The isotopy class $[c_n]$ represents a vertex in $\C(T)$; which is simply the Farey graph with vertex set $\QQ \cup \{\infty\}$ each point representing the slope of the curve.   By construction, as $n \to \infty$, the average slope of $j_n$, and hence $c_n$, approaches $\alpha$.   Thus, necessarily $[c_n]$ eventually leaves each compact set of $\C(T)$.  

We now show that this means $f$ cannot act elliptically on $\cd(T)$.  Suppose for contradiction that it did, i.e. $d^\dagger(\gamma, f^n(\gamma))\leq K$ for some constant $K\geq 0$.  Then $d^\dagger(\gamma', f^n(\gamma))\leq K+1$.  Let $P_n$ be a finite set of points so that each complementary region of $\gamma' \cap f^n(\gamma)$ contains at least one point of $P_n$.  Then 
$\gamma_n$ and $f^n(\gamma)$ are in minimal position in $T-P_n$, so by \cite[Lemma~3.4]{BHW} the distance in $\mathcal{C}^s(T-P)$ between $[\gamma_n]_{T-P_n}$ and $[f^n(\gamma)]_{T-P_n}$ is at most $K+1$ also. 

Since $c_n$ is a union of a subarc of $\gamma'$ and of $f^n(\gamma)$ and these are in minimal position in $T-P_n$,  the class $[c_n]_{T-P_n}$ is a bicorn of $[\gamma']_{T-P_n}$ and $[f^n(\gamma)]_{T-P_n}$.  In \cite{Rasmussen} A.~Rasmussen proves that the set of bicorns between two curves in the nonseparating curve graph of any finite-type surface satisfies the criterion of Masur--Schleimer \cite[Theorem 3.5]{MS}. In fact any bicorn will be a uniformly bounded distance $L\geq 0$ away from a geodesic \cite[Proposition 3.1]{Bowditch}. Therefore the distance between $[c_n]_{T-P_n}$ and  $[\gamma']_{T-P_n}$ is bounded by at most $K+L+1$, a bound independent of $n$. However, there is a $1$-coarsely Lipschitz 
 map of the (nonseparating) curve graph of $T-P_n$ to the Farey graph, simply by considering isotopy classes in $T$.  This contradicts our earlier observation that the sequence $[c_n]$ is unbounded in $\C(T)$, a contradiction. We conclude that $f$ acts parabolically on $\cd(T)$.\end{proof}

        \begin{remark}
          A modification of the construction (by post-composing with the time $t_0$-map $\varphi_{t_0} $ of the flow the
          homeomorphism $ f$) can be used to show that for any
          irrational $\alpha$ and any $t_0 \neq 0$ there is a
          homeomorphism acting parabolically with rotation set
          $ \{(t_0\alpha,t_0)\}$. 
        \end{remark}
    \begin{remark}[Parabolics with irrational slope]
          A further variation yields homeomorphisms whose rotation sets is a segment of irrational slope through the origin. Simply let $\chi\colon T \to [0,1]$ be a continuous bump function that is {\em surjective} and vanishes precisely on the suspension of the invariant Cantor $C$ of the Denjoy map. Then consider the time-$1$ map of the reparametrized Denjoy flow $\varphi_s(p)$ given by setting $s(t,p) = t(1-\chi(p))$.
          
          This flow has fixed points and agrees with the Denjoy flow on $C$. Hence one obtains points $(0,0)$ and $(1,\alpha)$ in the rotation set. Since $f$ preserves flow lines it is easy to see that all other elements of the rotation set are (positive) multiples of $(1,\alpha)$ and hence $\rho(f)$ is a segment of irrational slope.
        \end{remark}
%------------------------
\subsection{Proof of Theorem~\ref{thm:parabolicsexist}}
A similar construction can also be used to build homeomorphisms
of higher genus surfaces $S$ which act parabolically on $\cd(S)$, proving Theorem \ref{thm:parabolicsexist}. We give the details now, and discuss alternative constructions below.  

\begin{proof}[Proof of Theorem~\ref{thm:parabolicsexist}, first construction]
Let $S$ be a surface of genus $g\geq 2$, fix a hyperbolic structure on $S$, and let $\Lambda$ be a minimal filling geodesic lamination on $S$. 
Let $X$ be a vector field  
supported on the complement of $\Lambda$, so that the flow of $X$ pushes points into the cusps of the lamination.  This may be defined explicitly on each complementary region of $\Lambda$, modeled on a vertical flow supported on a standard ideal hyperbolic triangle with one vertex at infinity in the upper half plane.  Let $f$ be the time-one map of this flow.  Note that by cutting off the support of the flow, we may view $f$ as a $C^0$-limit of homeomorphisms supported on disks.

Let $\hat S \rightarrow S$ be a finite cover of $S$. Via pullback, $\hat S$ inherits a hyperbolic structure from $S$, and the preimage $\hat\Lambda$ is a minimal filling geodesic lamination of $\hat S$. We may lift our flow described above to $\hat S$ and write $\hat f$ for the time-one map. The minimality of $\hat\Lambda$ ensures the following: for any simple closed curves $\gamma$ and $\gamma'$ of $\hat S$, there exists $N$ such that for any $n>N$ we have that $\hat f^n \gamma$ intersects $\gamma'$. This is because $\gamma$ intersects the support of $\hat f$ (every half-leaf is dense in $\hat \Lambda$), moreover there is some subarc $c$ of $\gamma$ inside the support that connects different half-leaves of $\hat\Lambda$.  Deeper into the cusp, there is a subarc $c'$ of $\gamma'$ likewise connecting different half-leaves. By definition of $\hat f$, eventually $\hat f^n c$ intersects $c'$ for sufficiently large $n$.

It thus follows from Lemma~\ref{lem:small-diameter-one-cover} that any $f$-orbit in $\cd(S)$ has infinite diameter and so $f$ is not elliptic. On the other hand, $f$ is the $C^0$-limit of disk-supported maps. Since such maps
act elliptially on $\cd(S)$, $C^0$-continuity of the asymptotic translation length (Theorem~\ref{thm:continuity}) shows
that $f$ also has translation length $0$. This shows that it is parabolic.\end{proof} 

\begin{proof}[Proof of Theorem~\ref{thm:parabolicsexist}, second construction] Start with an abelian differential on $S$ and pick a slope $\lambda$ that defines a minimal (oriented) foliation $\mathcal{F}_\lambda$. The abelian differential endows $S$ with a translation surface structure, in particular a metric with zero curvature off the singularities, which we normalise with unit area. For simplicity we will assume that there are no saddle connections of slope $\lambda$.

Consider a smooth  vector field on $S$ with direction $\lambda$ with zeroes occurring only at the singularities of the abelian differential. Let $f$ be the time-$1$ map of the flow along the vector field.

We now prove that $f$ does not act hyperbolically on $\cd(S)$. To do this, we will find for any $n\in \NN$ a curve $\alpha$ such that $d^\dagger(\alpha,f^n \alpha)\leq 2$. Given $n\in \NN$, pick a (small) closed interval $J$ perpendicular to the slope $\lambda$, such that $\cup_{t\in [0,n]} \phi_t J$ is disjoint from the singularities, and such that the $\phi_t J$ are pairwise disjoint. The existence of such an interval follows from the fact that $\mathcal{F}_\lambda$ contains an infinite leaf. 
 Choose $\alpha$ to be a simple closed curve which is a union of an interval of slope $\lambda$, and a subinterval of $J$ (which is possible because there is a half-leaf of slope $\lambda$ through $J$ which is dense). Observe that the complement of $\alpha \cup f^n \alpha$ consists of two regions, one of which is a rectangle, foliated by the $\phi_t J$. It must be the case that the other region carries an essential simple closed curve disjoint from $\alpha \cup f^n \alpha$, and so we conclude that $d^\dagger(\alpha,f^n \alpha)\leq 2$ and thus $f$ is not hyperbolic.

To finish the proof we now prove that $f$ is not elliptic, by considering a curve containing a horizontal segment adjacent to a singularity.   For concreteness,   
define a curve $\alpha$ on $S$ by starting at a singularity $x$, following a line of slope $\lambda$ until one intersects a horizontal through the same singularity, then following that horizontal to $x$ in order to close the curve.  Thus, $\alpha$ is the union of a horizontal segment $\tau$, and a segment of slope $\lambda$ each of which have one endpoint at $x$.   
We claim that $d^\dagger(\alpha,f^n \alpha)$ is unbounded.  
To show this, suppose for contradiction that $d^\dagger(\alpha,f^n \alpha)\leq K$, for all $n\in\ZZ$. By Lemma~\ref{lem:small-diameter-one-cover} there is a finite-sheeted cover $\hat S$ of $S$ such that $\alpha$ and $f^n\alpha$ admit disjoint elevations, for any $n\in\ZZ$. 

Let $\mu$ be the half-leaf of $\mathcal{F}_\lambda$ emanating from $x$.  Note that every lift of $\mu$ to $S$ is dense in $S$.  
As $n$ increases, $f^n(\tau)$ converges on compact sets to $\mu$.   For $n$ sufficiently large, any elevation of $f^n(\tau)$ to $S$ will intersect every 
horizontal segment of fixed length, and thus intersects any elevation of $\tau$.  This contradicts that $\alpha$ and $f^n\alpha$ should have disjoint elevations, concluding the proof. \end{proof}

	\bibliographystyle{alpha}
	\bibliography{torus}
\end{document}